\documentclass[a4paper, 11pt]{amsart}
\usepackage{mathrsfs}
\usepackage{amsfonts, amssymb, amsxtra}
\usepackage{amsmath, bm}
\usepackage{dsfont}
\usepackage{color}
\usepackage{graphicx}
\usepackage[english,polish]{babel}
\usepackage{enumerate}
\usepackage{mathtools}
\usepackage{bbm}
\usepackage{enumitem}

\usepackage[compress, sort]{cite} 

\usepackage{newtxmath}
\usepackage{newtxtext}

\usepackage[margin=2.5cm, centering]{geometry}
\usepackage[colorlinks,citecolor=blue,urlcolor=blue,bookmarks=true]{hyperref}

\usepackage[utf8]{inputenc}

\newcommand{\CC}{\mathbb{C}}
\newcommand{\ZZ}{\mathbb{Z}}
\newcommand{\sS}{\mathbb{S}}
\newcommand{\NN}{\mathbb{N}}
\newcommand{\RR}{\mathbb{R}}

\newcommand{\calC}{\mathcal{C}}

\newcommand{\calB}{\mathcal{B}}

\newcommand{\calX}{\mathcal{X}}

\newcommand{\calD}{\mathcal{D}}
\newcommand{\calH}{\mathcal{H}}

\newcommand{\calJ}{\mathcal{J}}

\newcommand{\JLf}[1]{\ell_{#1}}

\newcommand{\Id}{\operatorname{I}}

\newcommand{\norm}[1]{\left\lVert {#1} \right\rVert}

\newcommand{\lnorm}[1]{\downharpoonleft\!\!\!| {#1} |\!\!\!\downharpoonright}

\newcommand{\barrier}{\mathfrak{b}}
\newcommand{\barrierMIN}{\barrier^{\mathrm{min}}}
\newcommand{\barrierTR}{\barrier^{\mathrm{TR}}}

\newcommand{\sprod}[2]{\left\langle {#1}, {#2} \right\rangle}

\newcommand{\tr}{\operatorname{tr}}

\newcommand{\sigmaP}{\sigma_{\mathrm{p}}}

\newcommand{\sigmaAC}{\sigma_{\mathrm{ac}}}
\newcommand{\sigmaS}{\sigma_{\mathrm{sing}}}

\newcommand{\Msing}{{M_{\mathrm{sing}}}}
\newcommand{\Mac}{{M_{\mathrm{ac}}}}

\newcommand{\Sac}{S_{\mathrm{ac}}}
\newcommand{\Sacr}{S_{\mathrm{ac}, \mathrm{r}}}
\newcommand{\Sacd}{S_{\mathrm{ac}, \mathrm{d}}}
\newcommand{\Ssing}{S_{\mathrm{sing}}}

\newcommand{\Dom}{\mathrm{Dom}}
\newcommand{\rank}{\mathrm{rank}}

\newcommand{\ud}{{\: \rm d}}

\newcommand{\cl}[1]{\operatorname{cl}(#1)}
\newcommand{\clleb}[1]
{\overline{#1}^{_\mathrm{e}}}

\renewcommand{\Im}{\operatorname{Im}}
\renewcommand{\Re}{\operatorname{Re}}

\newtheorem{theorem}{Theorem}
\newtheorem{lemma}[theorem]{Lemma}
\newtheorem{proposition}[theorem]{Proposition}
\newtheorem{fact}[theorem]{Fact}
\newtheorem{corollary}[theorem]{Corollary}

\theoremstyle{definition}
\newtheorem{example}[theorem]{Example}
\newtheorem{remark}[theorem]{Remark}
\newtheorem{definition}[theorem]{Definition}

\numberwithin{equation}{section}
\numberwithin{theorem}{section}

\DeclareMathOperator{\lin}{lin}

\newcommand{\bde}{\begin{description}}
\newcommand{\ede}{\end{description}}

\newcommand{\beql}[1]{\begin{equation}\label{#1}}
\newcommand{\eeq}{\end{equation}}

\newcommand{\kwa}{$\Box$}
\newcommand{\Kwa}{\hfill\kwa\bigskip}

\newcommand{\strz}{\longrightarrow}

\newcommand{\opcje}[4]{\left\{
    \begin{array}{ll}#1&\mbox{for $#2$}\\#3&\mbox{for $#4$}\end{array}\right.}

\newcommand{\N}{\NN}
\newcommand{\Nk}[1]{\N_{#1}}  
\newcommand{\Nz}{\Nk{0}}    
\newcommand{\Nmj}{\Nk{-1}}

\newcommand{\eps}{\epsilon}

\newcommand{\si}{\sigma}

\newcommand{\la}{\lambda}

\newcommand{\de}{\delta}

\newcommand{\om}{\omega}

\newcommand{\Mdc}{M_{d}(\CC)}
\newcommand{\Mddc}{M_{2d}(\CC)}

\newcommand{\lesc}{\left\langle} 
\newcommand{\risc}{\right\rangle}

\newcommand{\col}[2]{#1^{\{#2\}}} 

\newcommand{\el}{\ell}
\newcommand{\lnz}[1]{\el(\Nz, #1)}
\newcommand{\lnmj}[1]{\el(\Nk{-1}, #1)}
\newcommand{\lnk}[2]{\el(\Nk{#1}, #2)}

\newcommand{\ldnz}[1]{\el^2(\Nz, #1)}
\newcommand{\ldnmj}[1]{\el^2(\Nk{-1}, #1)}

\newcommand{\sprz}[1]{\overline{#1}}
\newcommand{\Bor}{\mathrm{Bor}}
\newcommand{\restr}[2]{#1\hspace{-1.0ex}  \restriction_{#2}}

\newcommand{\Ran}{\operatorname{Ran}}

\newcommand{\nic}[1]{} 

\newcommand{\BLW}{L(W)}

\newcommand{\GEV}[1]{\mathrm{GEV}_{-1}\hspace{-0.1em}\left(#1\right)}
\newcommand{\MGEV}[1]{\mathrm{MGEV}_{-1}\hspace{-0.1em}\left(#1\right)}
\newcommand{\NMGEV}[1]{\mathrm{MGEV}_{\hspace{-0.1em}\mathrm{nor}\hspace{0.1em}}\hspace{-0.1em}\left(#1\right)}

\newcommand{\GEVze}[1]{\mathrm{GEV}\hspace{-0.1em}\left(#1\right)}
\newcommand{\MGEVze}[1]{\mathrm{MGEV}\hspace{-0.1em}\left(#1\right)}

\newcommand{\GEVld}[1]{\mathrm{GEV}_{\ell^2}\hspace{-0.2em}\left(#1\right)}

\newcommand{\Caj}{{\calC}_J}
\newcommand{\Wej}{W}

\title{Barrier nonsubordinacy and absolutely continuous spectrum of block Jacobi matrices}
\author{Marcin Moszy\'{n}ski}
\address{
    Marcin Moszy\'{n}ski \\
    Faculty of Mathematics, Informatics and Mechanics \\
    University of Warsaw \\
    ul. Stefana Banacha 2 \\
    02-097 Warsaw, Poland
}
\email{mmoszyns@mimuw.edu.pl}

\author{Grzegorz \'{S}widerski}
\address{
    Grzegorz \'{S}widerski \\
    Institute of Mathematics \\ 
    Polish Academy of Sciences \\
    ul. Śniadeckich 8 \\
    00-696 Warsaw, Poland
}
\curraddr{Faculty of Pure and Applied Mathematics, Wroclaw University of Science and Technology, Wyb. Wyspiańskiego 27, 50-370 Wroclaw, Poland}
\email{grzegorz.swiderski@pwr.edu.pl}

\keywords{Block Jacobi matrix, absolutely continuous spectrum, matrix measures, Weyl function}
\subjclass[2020]{Primary 47B36}

\begin{document}
\selectlanguage{english}

\begin{abstract}
 We explore to what extent the relation between 
  the absolute continuous spectrum  and  non-existence of subordinate generalized eigenvectors, known for scalar Jacobi operators, can be formulated  also for block Jacobi operators with $d$-dimensional blocks. 
   The main object here  allowing  to make some progress  in that direction is the new notion of the barrier nonsubordinacy.
  We prove that the barrier nonsubordinacy implies the absolute continuity for block Jacobi operators.
   Finally, we extend  some well-known  $d=1$ conditions guaranteeing the absolute continuity to $d \geq 1$ and we give  applications of our results to some concrete classes of block Jacobi matrices.
\end{abstract}

\maketitle

\section{Introduction}
A \emph{block Jacobi matrix} is a semi-infinite block tridiagonal Hermitian matrix of the form
 \begin{equation} \label{eq:127'}
     \calJ =   
     \begin{pmatrix}
         B_0	& A_0	&		&       \\
         A_0^*	& B_1	& A_1	&         \\
             & A_1^*	& B_2 	& \ddots   \\
             &		&		\ddots &  \ddots 
     \end{pmatrix},
 \end{equation}
where $A_n$ and $B_n$  are ''blocks", i.e., $d \times d$ complex matrices with $A_n$ invertible and Hermitian $B_n$ for all $n \in \NN_0 = \{0,1,2,\ldots \}$. If $d=1$ we shall refer to $\calJ$ as \emph{the (scalar) Jacobi matrix}.
The action of $\calJ$ is well-defined on the linear space $\lnk{0}{\CC^d}$ of all $\CC^d$-valued sequences, and it is natural to define the operator $J$ as the restriction of $\calJ$ to the Hilbert space $\ldnz{\CC^d}$, where
\[
     \ldnz{\CC^d} = 
     \bigg\{ 
     	x \in \lnk{0}{\CC^d} : 
		\sum_{n=0}^{+\infty} \norm{x_n}_{\CC^d}^2 < \infty 
     \bigg\}.
\]
Namely, we define $J x = \calJ x$ for any $x \in \ldnz{\CC^d}$ such that also $\calJ x \in \ldnz{\CC^d}$. The operator $J$ is called \emph{the (maximal) block Jacobi operator}. It does not have to be self-adjoint. If it is, then one can define a matrix of Borel complex measures on the real line by the formula
\begin{equation} \label{eq:149'}
	M(G) = 
	\big( 
	\langle E_J(G) \delta_0(e_j), \delta_0(e_i) \rangle 
	\big)_{i,j=1 \ldots d} \qquad G \in \Bor(\RR),
\end{equation}
where $E_J$ is the projection-valued spectral measure of $J$ and for any $v \in \CC^d$ the sequence $\delta_0(v)$ is  equal to $v$ on $0$th position and $0$ elsewhere; by $\{e_1,\ldots,e_d\}$ we denote the standard orthonormal basis of $\CC^d$. Then $J$ turns out to be unitary equivalent to the operator acting by the multiplication by the identity function on the space $L^2(M)$ of classes of $\CC^d$-valued functions, see,  e.g., \cite{Weyl, Moszynski2022} for details. The well-known sufficient condition for the self-adjointness of $J$ is the \emph{Carleman condition}
\begin{equation} \label{eq:128'}
	\sum_{n=0}^{+\infty} \frac{1}{\norm{A_n}} = \infty,
\end{equation}
see e.g. \cite[Theorem VII-2.9]{Berezanskii1968}.

The interest in block Jacobi operators comes from mathematical physics and operator theory (see e.g. \cite{Aptekarev1984, Gorbachuk1997}), their close relation to the matrix moment problem (see e.g. \cite{Duran2001a, Berg2008, Krein1949a, Krein1949b}) as well as from the theory of matrix orthogonal polynomials on the real line, see e.g. \cite{Damanik2008, Sinap1995}, and Pad\'{e} approximations and continued fractions (see e.g. \cite{Beckermann1997, Zygmunt2002}). They are useful for analysis of difference equations of finite order, see \cite{Duran1995}. Some types of block Jacobi operators are related to random walks and level dependent quasi--birth--death processes, see e.g. \cite{Dette2007}. For further applications we refer to \cite{Sinap1996}.

Spectral analysis of block Jacobi operators for $d > 1$ is not well-developed yet.  In \cite{Sahbani2016} the Mourre's commutator method was applied to study compact perturbations of constant sequences $(A_n)_{n\geq0}$, $(B_n)_{n\geq0}$. Under some regularity hypotheses it was shown there that the singular continuous spectrum is empty. A commutator method was also applied in \cite{Janas2018, Janas2020} for obtaining a bound on the entries of the resolvent. In \cite{block2018}, by analysing Tur\'an determinants, the continuous spectrum of bounded and unbounded Jacobi operators was studied.
 Next, in \cite{Kupin2018}  the problem of localization of the essential spectrum of unbounded block Jacobi operators  was studied by estimating the quadratic form associated to $J$. Finally, in a recent article \cite{Oliveira2022}, some adaptations of scalar methods from \cite{Jitomirskaya1999, LastSimon1999} to the case $d \geq 1$ were established.

In this article we shall give sufficient conditions for absolute continuity\footnote{In this article the absolute continuity will be always understood with respect to the Lebesgue measure.} of the matrix measure $M$ on some regions of the real line. For doing so, we extend some parts of Pearson--Khan subordinacy theory for scalar Jacobi matrices (see \cite{Khan1992}). Subordinacy theory was initially formulated for $1$-dimensional Schr\"odinger operators on the half-line by Gilbert--Pearson in \cite{Gilbert1987} and later adapted to several classes of operators, see e.g. the excellent survey \cite{Gilbert2005}.

For  $\la\in\CC$ we say that a sequence $U = (U_n)_{n \in \NN_0}$ of $d \times d$ matrices is a \emph{matrix generalized eigenvector ( ``mgev'' for short) \  for $J$ and $\la$}, \ \ $U \in \MGEVze{\la}$, \ if it satisfies the recurrence relation
\begin{equation} \label{eq:129'}
	\la U_n = A_{n-1}^* U_{n-1} + B_n U_n + A_n U_{n+1}, \quad n \geq 1
\end{equation}
with some values $(U_0, U_1)$. For our purposes it will often be somewhat more useful to extrapolate sequences $U \in \MGEVze{\la}$ to $\NN_{-1}=\NN_0\cup\{-1\}$, in accordance with the common manner. So,  we 
define\footnote{By $\Id$ we denote the identity matrix of the appropriate dimension and also the identity operator in any liner space.}  $A_{-1} := \Id$ \ and require \eqref{eq:129'} also for $n=0$. Then  such  $U$ being already extrapolated mgev is called   {\em emgev} here,  and we can 
 consider the initial conditions $(U_{-1}, U_0)$  for it,  instead of $(U_{0}, U_1)$.
In what follows the two emgev-s $P(\lambda), Q(\lambda)$ corresponding to the initial conditions
\begin{equation} 
	\label{eq:140}
	\begin{cases}
		Q_{-1}(\lambda) = \Id, \\
		Q_0(\lambda) = 0
	\end{cases} \qquad
	\begin{cases}
		P_{-1}(\lambda) = 0, \\
		P_0(\lambda) = \Id
	\end{cases}
\end{equation}
will be instrumental. 

For any sequence $X=(X_k)_{k\geq-1}$ \ of $d \times d$ matrices  and any $n\in \Nz$ we define
\begin{equation} \label{eq:131'}
	\norm{X}_{[0,n]} = \sqrt{\sum_{k=0}^n \norm{X_k}^2},
\end{equation}
where we use the operator norm for matrices. 
\begin{definition}\label{denonsub}
We shall say that $J$ satisfies the \emph{\textbf{non}subordinacy condition} for some $\lambda \in \RR$,  if
for any two non-zero emgev-s  $U,V$ for $J$ and $\la$
\begin{equation} \label{eq:130'}
	\liminf_{n \to \infty} \frac{\norm{U}_{[0,n]}}{\norm{V}_{[0,n]}} < \infty. 
\end{equation}
\end{definition}
Recall that in the scalar case  a non-zero emgev  $V$ for $\la$ is \emph{subordinate}, if for each linearly independent emgev  $U$ for $\la$ the left-hand side of \eqref{eq:130'} is equal to infinity. The main result of Khan--Pearson in \cite{Khan1992} is their Theorem 3, which implies that
the set where $J$ satisfies nonsubordinacy condition is a minimal support of the absolutely continuous part of $M$, whereas 
the set where $P$ is subordinate is a minimal support of the singular part of $M$. The minimality, as well as the absolute continuity and the singularity are  understood here with respect to the Lebesgue measure.

In this article we extend some parts of \cite[Theorem 3]{Khan1992} to the setup of block Jacobi matrices with $d$-dimensional blocks. Note that such  generalizations to the case of dimension $d > 1$ encountered many serious problems, still unsolved. Here, to make some progress, we shall adapt some techniques of Jitomirskaya--Last (see \cite{Jitomirskaya1999}) to the present setup. 
For any $t > 0$ we can define $\norm{X}_{[0,t]}$ by an interpolation of \eqref{eq:131'}, see \eqref{eq:105a} for precise formulation. Then we can introduce a more quantitative version of Definition \ref{denonsub}. To do so we need the following definition.
\begin{definition} \label{def:barrier}
let $G \subset \RR$ be non-empty. We say that a function $\barrier : G \times [1, +\infty) \to \RR$ is a \emph{barrier} if for any $\la \in G$ and any two emgev   $U,V$ for $J$ and $\lambda$ normalized by 
\begin{equation} \label{eq:144}
	\norm{U_{-1}}^2 + \norm{U_{0}}^2 = \norm{V_{-1}}^2 + \norm{V_0}^2 = 1
\end{equation}
we have
\[
	\left(\frac{\norm{U}_{[0,t]}}{\norm{V}_{[0,t]}}\right)^2 \leq \barrier(\la, t), \quad t \geq 1.
\]
\end{definition}
Given a barrier $\barrier$ we say that $J$ is \emph{$\barrier$-nonsubordinate on $G$} if
\begin{equation} \label{eq:132}
	\liminf_{t \to \infty} \barrier(\la, t) < +\infty, \quad \la \in G.
\end{equation}
If this condition holds uniformly on $G$, namely   \ 
$    \sup_{\la \in G}
    \liminf_{t \to +\infty}
    \barrier(\la,t) < +\infty$, \ 
then we say that $J$ is \emph{uniformly $\barrier$-nonsubordinate} on $G$.

To the authors' knowledge the notion of a barrier has not appeared in the literature before, however, it was inspired by more restrictive Weidmann's notion of uniform nonsubordinacy from \cite[Section 3]{Weidmann1996} for $d=1$.

\begin{theorem} \label{thm:10-intr} (see Theorem~\ref{thm:10}) \ \
Assume that $J$ is self-adjoint,   $G \in \Bor(\RR)$ and $\barrier$ is a barrier for $J$ on $G$. If $J$ is $\barrier$-nonsubordinate on $G$, then
\begin{enumerate}[label=(\alph*)]
     \item $M$ defined in \eqref{eq:149'} is absolutely continuous on $G$. \label{thm:10-intr:a}
     \item there exists a density $D$ of $M$ on $G$ which is   an invertible matrix a.e. on $G$. \label{thm:10-intr:b}
     \item $J$ is absolutely continuous in $G$ and the  essential closure of $G$ is contained in $\sigmaAC(J)$.  \label{thm:10-intr:c}
\end{enumerate}
If, moreover, $J$ is uniformly $\barrier$-nonsubordinate on $G$, then there exist  $c_1, c_2>0$ such that the above density satisfies
\begin{equation} \label{eq:145}
     c_1 \Id \leq D(\la) \leq c_2 \Id
\end{equation}
for a.e.  $\la \in G$.
\end{theorem}

We have to emphasize that Theorem~\ref{thm:10-intr}, unlike Khan--Pearson theory for  the $d=1$ case, gives only a sufficient condition for the absolute continuity of $M$. As we constructively show in Example~\ref{ex:2}, there exist block Jacobi operators such that for any bounded $G \subset \RR$ the properties: \ref{thm:10-intr:a}, \ref{thm:10-intr:b}, \ref{thm:10-intr:c} and \eqref{eq:145} are satisfied, but $J$ does not satisfy \eqref{eq:130'} for any $\la \in \RR$. Nevertheless, as we show in Theorem~\ref{thm:11}, we can use Theorem~\ref{thm:10-intr} to prove that some classes of Jacobi operators considered in \cite{block2018} are absolutely continuous, which was stated there as an open problem.

In Section~\ref{sec:6} we adapt to our setup some of the sufficient conditions implying non-existence of subordinate solutions, which proved to be useful in the case $d=1$. In particular, the one introduced by Last--Simon in \cite{LastSimon1999}. These conditions are formulated in terms of   transfer matrices. Recently, by adapting the approach of \cite{LastSimon1999}, an analogous conclusion for $d \geq 1$ was established in \cite[Theorem 1.2]{Oliveira2022} under the assumption that $A_n = A_n^*$ for any $n$ and
 \begin{equation} \label{eq:139}
     \sup_{n \geq 0} \big(\norm{A_n} + \norm{A_n^{-1}} \big) < \infty.
 \end{equation}
In contrast, we show in Theorem~\ref{thm:5} that \emph{Generalized Last--Simon} (GLS in short) condition implies the hypotheses of Theorem~\ref{thm:10-intr} for a suitably chosen barrier $\barrier$ (see \eqref{eq:120}), which is a different approach from the one used in \cite{Oliveira2022}. Let us stress that GLS condition is substantially stronger than the hypotheses of Theorem~\ref{thm:10-intr}. Indeed, as we show in Proposition~\ref{prop:20}, the Carleman condition \eqref{eq:128'} is necessary for GLS condition to hold.

Our approach to prove Theorem~\ref{thm:10-intr} is based on careful extension of the methods of Jitomirskaya--Last linking the asymptotic properties of generalized eigenvectors to the boundary values of the Cauchy transform of the matrix measure $M$. Along the way, we derive an analogue of their famous inequality \cite[Theorem 1.1]{Jitomirskaya1999}, see Theorem~\ref{thm:9}. Recently, Oliveira--Carvalho in \cite[Theorem 1.1]{Oliveira2022} also obtained an extension of Jitomirskaya--Last's inequality to $d \geq 1$ which is apparently too weak for proving our Theorem~\ref{thm:10-intr}.

The article is organized as follows. In Section~\ref{sec:2} we fix our notation and summarize  some fundamental results on block Jacobi matrices we need here. For some additional background with detailed proofs we refer to \cite{Weyl}. In Section~\ref{sec:5} we prove the central results: Theorem~\ref{thm:9}  --- the main result of the paper and its spectral consequence: Theorem~\ref{thm:10} (cf. Theorem~\ref{thm:10-intr}).  Section~\ref{sec:6} extends some well-known conditions implying nonsubordinacy from $d=1$ to the general case. In particular, it covers: GLS (Generalized Last--Simon) condition, GBS (Generalized Behncke--Stolz) condition and $H$ (homogenous) class condition. Finally, in Section~\ref{sec:7}, we show some examples and counterexamples illustrating the applicability of our results.

 \subsection*{Acknowledgment}
The article was partially supported by grant ``Subordinacy for block Jacobi operators. Spectral theory for self-adjoint finitely-cyclic
operators and the introduction to $L^2$ type matrix measure spaces'', NI 3B  POB III  IDUB (01/IDUB/2019/94),   funded by University of Warsaw, Poland.
 The  author Grzegorz Świderski was also partially supported by long term structural funding -- Methusalem grant of the Flemish Government. Part of this work was done while he was a postdoctoral fellow at KU Leuven.
 The author  Marcin Moszyński wishes to thank: Anna Moszyńska (IPEVP \& MusInvEv,  Warsaw) -- his wife --
 for  extraordinary patience and for valuable linguistic help; \  Nadia V. Zaleska (EIMI, St. Petersburg \& MusInvEv,  Warsaw) -- his friend --
  for some wise  hints and for invaluable moral support; \ 
 Grzegorz Świderski -- the co-author --
 for several years of   confidence in success and for the idea of {\em  barriers}.

\section{Preliminaries} \label{sec:2}

We collect and fix here some  general notation and terminology for the paper,  which will be important   in the main  sections.

\subsection{Some abbreviations and common symbols} \label{abb+symb}

We use  here also  some  ``more or less common'' {  abbreviations} and { symbols} like: \  iff \ (if and only if), \
TFCAE \ (the following conditions are [mutually] equivalent), \ w.r.t. \ (with respect to), \  s.a. \  (self-adjoint),  \ a.c. \ (absolutely continuous), \ 
 sing \ (singular), \  
a.e. \ (almost everywhere), \ 
JM, BJM \ (Jacobi matrix, block Jacobi matrix), \ JO, BJO\ (Jacobi operator, block Jacobi operator),  \ 
 $\lin Y$ \ (the linear subspace generated by a subset $Y$ of a linear space),  \ $\restr{F}{Y}$ \ (the restriction of  function $F$ to the subset $Y$ of the domain), \ 
$F(Y)$ \ (the image of  subset $Y$ with respect to function $F$), \   
$\Bor(\RR)$ \ (the Borel  $\si$-algebra of $\RR$), \   $\calB(X)$ \  (the space of bounded operators on a normed space $X$).

\subsection{Introductory notation and notions} \label{intr-not}

We   use here the following symbols for some sets of scalars: 
\[
    \CC_+:= \{ z \in \CC : \Im(z) > 0 \},\hspace{1.5em} 
    \RR_+ := \{ t \in \RR: \ t > 0 \},\hspace{1.5em}  
    \N_{k} := \{ n \in \ZZ: n\geq k \} \ \ \ \mbox{ for\ } k\in\ZZ,
\]
so, e.g.,  
$
    \N = \Nk{1}, \quad  
    \Nz = \N\cup\{0\}, \quad 
    \Nmj = \N\cup\{-1,0\}
$.
Let  $d\in\N$. By  \ 
$
e_1, \ldots,e_d
$
\
 denote the vectors of the standard base in $\CC^d$, and by \ $M_d(\CC)$ \  the space of all  $d \times d$ complex matrices,  with the usual matrix/operator norm.  We identify any $A\in M_d(\CC)$ with the    linear transformation of $\CC^d$ induced by matrix $A$ and any $v \in \CC^d$ with the column vector.
\emph{The real and imaginary parts} $\Re(A)$ and  $\Im(A)$ (in the adjoint, and not in the complex conjugation  sense)  are given by 
\begin{equation} \label{eq:150}
    \Re(A) : = \tfrac{1}{2} (A + A^*), \quad 
    \Im(A): = \tfrac{1}{2i} (A - A^*).
\end{equation}
For $v\in\CC^d$ we shall use the symbol $E^v$ to  denote the matrix determined  by 
\beql{def:Ev}
    E^v(e_j) = \opcje{v}{j=1}{0}{j>1,} \quad j=1,\ldots, d.
\eeq
The symbol $\col{A}{j}$ denotes  the   $j$th column of matrix $A$ (treated as a $\CC^d$-vector) and for vectors $v^{(1)},\dots,v^{(d)}\in\CC^d$  the matrix $A$ with  $\col{A}{j}=v^{(j)}$\ \ for any $j$ is denoted  by $[v^{(1)},\dots,v^{(d)}]$.

The symbols $\norm{\cdot}_X$ \ and  \  $\sprod{\cdot}{\cdot}_{X}$ denote here the norm in  a normed space $X$ and the scalar product in a Hilbert space $X$, respectively, but  we sometimes ``simplify'' the subscript $X$ for some $X$-s having a more complicated  symbol,  and  we  generally   omit it  for all operator norms,  which we use by default  for the bounded operators (mainly for matrices from $\Mdc$),  if no other choice is made. 
 For  a linear  operator $A : X \strz X$ in a normed space $X \neq \{0\}$  we define its \emph{minimum modulus} by
\begin{equation}
	\label{eq:23}
	\lnorm{A} := \inf_{\| x \|_{X} = 1} \| A x \|_X.
\end{equation}

We use  \ 
$  \cl{G}
$ \ 
for the closure of a subset $G$ of a topological space ($\sprz{{\la}}$ 
is   the complex conjugation of $\la\in\CC$),  and 
\ $\clleb{G}$ \ 
for the essential closure of  $G\in\Bor(\RR)$ w.r.t.    the Lebesgue measure $|\hspace{0.01em}\cdot\hspace{0.01em}|$ on  $\Bor(\RR)$, i.e., 
\ $\clleb{G}:=  
	\{ 
		t \in \RR : \forall_{\varepsilon > 0} \
		|G \cap (t-\varepsilon; t+ \varepsilon)| > 0
	\}
 $ (see e.g. \cite{Gesztesy2008}).  For linear space $X$,  by \  $\lnk{k}{X}$ \ we denote the linear space of {\bf all} the sequences $x=(x_n)_{n\in\Nk{k}}$ with terms in $X$.
If $X$ is a normed space, then 
\[
    \ldnz{X}:= 
    \bigg\{ x \in \lnz{X} : \sum_{n=0}^{+\infty} \| x_n \|_{X}^2 < \infty \bigg\}
\]
 is a normed space with the norm   \ $
    \norm{x}_{\el^2} := \sqrt{\sum_{n=0}^{+\infty} \| x_n \|_{X}^2}$, for $x\in\ldnz{X}$.
If $X$ is a Hilbert space, then $\ldnz{X}$ is a Hilbert space  with the scalar product given for $x, y\in\ldnz{X}$ by
\beql{l2scal}
	\langle x, y \rangle_{\el^2} := \sum_{n=0}^{+\infty} \langle x_n, y_n \rangle_{X}.
\eeq

Here, the most important  is the space $\ldnz{\CC^d}$ with 
the standard orthonormal basis
\[
    \big\{ \delta_n(e_i) \big\}_{(i,n) \in \{1,\ldots,d\} \times \Nz},
\]
where  for any vector $v \in \CC^d$ and $n\in\Nz$ we define the sequence $\delta_n(v) \in \lnk{0}{\CC^d}$ by
\[
    (\delta_n(v))_m :=
    \begin{cases}
	v & \text{if } m=n \\
	0 & \text{otherwise,}
    \end{cases} \qquad \qquad n,m\in\Nz.
\]

If  $\calH$ is  a Hilbert space and $A$ is s.a. in $\calH$ (including the unbounded case), then  the projection-valued spectral measure  of $A$  is  denoted by $E_A$, \  $E_A : \Bor(\RR) \strz \calB(\calH)$. 
If $x,y \in \calH$,  then $E_{A,x,y}$ is the complex measure  given by \ $
    E_{A,x,y}(\om) := \lesc E_A(\om) x, y \risc_{\calH}$ for $\om\in  \Bor(\RR)$
and $E_{A,x}:=E_{A,x,x}$.
Denote also \ $\calH_{\mathrm{ac}}(A) :=
	\{ x \in \calH : E_{A,x} \text{ is a.c. w.r.t.}  \ |\hspace{0.01em}\cdot\hspace{0.01em}|\}$. Here, for  $G \in \Bor(\RR)$, the key spectral notion  is:  
\emph{$A$ is absolutely continuous (a.c.) in $G$}, which means that  $\Ran E_A(G) \subset \calH_{\mathrm{ac}}(A)$.
And \ 
\emph{$A$ is absolutely continuous} \ \  iff \  $A$ is absolutely continuous in $\RR$.

\subsection{Block Jacobi operator. Finite cyclicity and the spectral matrix measure} \label{sec:3n}

We denote ``the size of the block'' for BJM by    $d$, $d\in\NN$, and for the whole  paper we assume that  $(A_n)_{n \in \Nz}$ and $(B_n)_{n \in \Nz}$ are  sequences of matrices from $M_d(\CC)$ such that 
\beql{zal-odw+sym}
    \det A_n \neq 0, \ \ \  
    B_n = B_n^*, \quad n \in \Nz.
\eeq
Define a \emph{block Jacobi matrix},  $\calJ$, by the formula~\eqref{eq:127'}.
The pair  $(A_n)_{n \in \NN_0}$, $(B_n)_{n \in \NN_0}$ is called \emph{Jacobi parameters of $\calJ$}, and 
 we treat $\calJ$ as the linear operator ({\em ``the formal block Jacobi operator''})  \  
$
    \calJ : \ell(\Nz, \CC^d)\strz \ell(\Nz, \CC^d) 
$,  \ 
such that   for any  $u\in\ell(\Nz, \CC^d)$
\beql{eq:defformal}
    \left(\calJ u\right)_n :=
    \opcje{ B_0 u_0 + A_0 u_{1}}{n=0}{	A_{n-1}^* u_{n-1} + B_n u_n + A_n u_{n+1}}{n\in\N,} 
\eeq
which gives
\beql{Jnadel}
    \calJ\left( \delta_n(v)\right) =
    \delta_{n-1}(A_{n-1}v)+\delta_{n}(B_{n}v)+\delta_{n+1}(A_{n}^*v), \qquad 
    v \in \CC^d,\  n\in\Nz, 
\eeq
where we additionally denote
$\delta_{-1}(v) := 0$ for any $v \in \CC^d$.
 Our main operator  is  
$J$ ({\em the maximal block Jacobi operator}), called here shortly {\em block Jacobi operator} (BJO),  is defined as  \  $J:=\restr{\calJ}{\Dom(J)}$, where  \ $\Dom(J) := \big\{ u\in\ldnz{\CC^d}: \calJ u\in\ldnz{\CC^d} \big\}$.
It is known, that (see,  e.g., \cite{Moszynski2022}, \cite{Weyl})
\begin{proposition} \label{prop-fincyc}
The system 
\beql{fix-fincyc}
    \vec{\varphi} := (\varphi_1, \ldots, \varphi_d), \quad
    \varphi_j := \de_{0}(e_j), \ \  j=1,\ldots,d\ \ \eeq
is a cyclic system for $J$.
\end{proposition}
Recall that the notion of {\em  cyclic system} (and of so-called finite-cyclicity) is a generalization of cyclic vector and the above assertion means just that  \ 
$
\lin\{J^n\varphi_j: \ n\in\Nz, \ j=1,\ldots,d\} $ \ 
is dense in $\ldnz{\CC^d}$.
Such a  choice of  cyclic system for $J$ is not unique, but  the particular  $\vec{\varphi}$ from  \eqref{fix-fincyc} is called {\em the canonical cyclic  system  for $J$}.  Similarly, instead of spectral (scalar) measure considered   for $J$ in the   $d=1$ case, for BJO  we  need an    abstract notion of {\em matrix measure}  and of  its particular choice for $J$ 
(see, e.g., \cite{Moszynski2022}, \cite{Weyl} or \cite[Section 8]{Weidmann1987}). 
\begin{definition}\label{smm}
\emph{The  spectral matrix measure $M=M_J= E_{J,\vec{\varphi}}$   \  for    $J$} is given by  
\begin{equation}
    \label{eq:IV:2.1}
    M: \Bor(\RR) \to \Mdc, \quad
    M(\om):= 
    \Big(E_{J,\varphi_j, \varphi_i}(\om)  \Big)_{i,j=1,\ldots,d}\in \Mdc  , \quad 
    \om \in \Bor(\RR). 
\end{equation}
 \end{definition}
We use rather $M$ and not   $M_J$ in this paper to denote the  spectral matrix measure   for    $J$,  when $J$ is fixed.

\subsection{The two associated difference equations and ``the solution extensions to -1''}
\label{sec:twoeq}

For any $z \in \CC$ let us consider  two 
difference equations tightly related to $\calJ$. The first --- ``the vector'' one --- is the infinite system of equations for a sequence $u=(u_n)_{n\in\Nz}\in\lnz{\CC^d}$:

\begin{equation} \label{eq:2}
    (\calJ u)_n = z u_n, \qquad n \geq 1.
\end{equation}
Each such a vector sequence $u$ is  called {\em generalized eigenvector (for $J$ and $z$)}\footnote{Note here, that it is ``generalized'' for two reasons; the first, because $u$ may not belong to $\Dom(J)$,  not even $\ldnz{\CC^d}$, and the second, since we do not require the equality  for $n=0$ above.} ---  {\em ``gev''} for short. By \eqref{eq:defformal} equivalently its  explicit form can be  written
\beql{eq:2'}
    A_{n-1}^* u_{n-1} + B_n u_n + A_n u_{n+1} = z u_n, \qquad n \geq 1.
\eeq
The second difference equation --- ``the matricial'' one --- is the analog equation for a matrix sequence  
$U = (U_n)_{n \in \Nz} \in \lnz{\Mdc}$:
\begin{equation} \label{eq:1}
    A_{n-1}^* U_{n-1} + B_n U_n + A_n U_{n+1} = z U_n, \qquad n \geq 1,
\end{equation}
and each such a matrix sequence $U$ is  called {\em matrix generalized eigenvector (for $J$ and $z$)} ---  {\em ``mgev''} for short. Having our BJM $\calJ$ fixed, for $z\in\CC$ we denote
\  $\GEVze{z} := \{u\in\lnz{\CC^d}:\, \mbox{$u$ is  a gev for $J$ and $z$}\}
$
and parallelly \ 
$\MGEVze{z} := \{ U \in \lnz{\Mdc}:\, \mbox{$U$ is  a  mgev for $J$ and $z$} \}
$, 
being obviously linear subspaces of  $\lnz{\CC^d}$ and $\lnz{\Mdc}$, respectively.
Denote also \ 
$
    \GEVld{z} := \GEVze{z} \cap \ldnz{\CC^d}
$.
It is often more convenient to use ``initial conditions at $-1$ and $0$'' instead of $0$ and $1$. To formulate this properly,  we shall define first the appropriate extension of each solution (in both, vector and matrix cases), which is tightly related to the  ``a priori choice'' of $A_{-1}$:
\beql{A-1}
    A_{-1} := -\Id.
\eeq
Taking into account \eqref{A-1}, let us consider ``extensions''  of the systems  \eqref{eq:2'} and \eqref{eq:1}:
\beql{eq:2'-1}
    A_{n-1}^* u_{n-1} + B_n u_n + A_n u_{n+1} = z u_n, \qquad n \geq {\pmb 0}
\eeq
for   sequences $u=(u_n)_{n \in \Nk{-1}} \in \lnk{-1}{\CC^d}$ and 
\beql{eq:1-1}
    A_{n-1}^* U_{n-1} + B_n U_n + A_n U_{n+1} = z U_n, \qquad n \geq {\pmb 0}
\eeq
for   sequences $U=(U_n)_{n \in \Nk{-1}} \in \lnk{-1}{\Mdc}$. 
Their solutions will be called {\em extended
 generalized eigenvectors} and  
{\em extended 
 matrix generalized eigenvectors}, respectively, {\em (for $J$ and $z$)} 
  ---  {\em ``egev''} and  {\em ``emgev''} for short. 
We denote also \ 
$  
    \GEV{z} := \big\{ u \in \lnk{-1}{\CC^d}:\, \mbox{$u$ is  an  egev for $J$ and $z$} \big\}
$ \ 
and  \ 
$
    \MGEV{z} := \big\{ U \in \lnk{-1}{\Mdc}:\, \mbox{$U$ is  an emgev for $J$ and $z$} \big\}.
$
For the matrix case, for any $z \in \CC$ choose
$Q(z), P(z) \in \MGEV{z}$  corresponding to \eqref{eq:140}
with the following general  notation:
for any  sequence $U(p) = ((U(p))_n)_{n \in \Nk{k}}$ depending on an extra ``function variable type parameter''  $p$: \ \ 
$
    U_n(p) := (U(p))_n
$ 
for any  $p$ and $n \in \Nk{k}$.

\subsection{Transfer matrices}
\label{subs-trans}

For  $z \in \CC$  the  generalized eigenequation \eqref{eq:2},  its matrix analog \eqref{eq:1},  as well as their extended to $-1$ variants, can be  written in  equivalent forms with the use of the so-called (one step) \emph{transfer matrices (for $J$ and $z$)}. The $n$th transfer matrix  $T_n(z) \in \Mddc$ has  the block form, with blocks in $\Mdc$: 
\begin{equation}
    \label{eq:21}
    T_n(z):=
    \begin{pmatrix}
        0 & \Id \\
        -A_{n}^{-1} A_{n-1}^* & A_n^{-1} (z \Id - B_n)
    \end{pmatrix},  \qquad n \geq 0 
\end{equation}
(for $n = 0$ recall  \eqref{A-1}).  So,  \eqref{eq:2'} \  (\eqref{eq:2'-1}) and also \eqref{eq:1} (\eqref{eq:1-1}) are  respectively equivalent to 
\begin{equation}
    \label{eq:20'}
    \begin{pmatrix}
        u_n \\
        u_{n+1}
    \end{pmatrix} 
    =
    T_n(z) 
    \begin{pmatrix}
        u_{n-1} \\
        u_n
    \end{pmatrix}, \quad \begin{pmatrix}
        U_n \\
        U_{n+1}
    \end{pmatrix} 
    =
    T_n(z) 
    \begin{pmatrix}
        U_{n-1} \\
        U_n
    \end{pmatrix},
     \qquad\quad n \geq 1  \ (\geq 0).
\end{equation} 
Define also  {\em the $n$-step transfer matrix} by 
\begin{equation} \label{eq:55}
    R_n(z) = T_{n-1}(z)\ldots T_0(z),  \qquad n\geq 1.
\end{equation}
This name is justified, e.g., by the property
\begin{equation} \label{eq:56}
    \begin{pmatrix}
        U_{n-1} \\
        U_{n}
    \end{pmatrix}
    =
    R_n(z) 
    \begin{pmatrix}
        U_{-1} \\
        U_0
    \end{pmatrix},
    \quad n \geq 1, 
\end{equation}
which we obtain  from \eqref{eq:20'}. Hence, by   \eqref{eq:56} and \eqref{eq:140} we get
\begin{equation} \label{eq:84}
    R_n(z) =
    R_n(z)\begin{pmatrix}
        \Id & 0 \\
        0 & \Id
    \end{pmatrix} = 
    \begin{pmatrix}
        Q_{n-1}(z) & P_{n-1}(z) \\
        Q_n(z) & P_n(z)
    \end{pmatrix}, \qquad n \geq 1.
\end{equation}

\subsection{$\ell^2$ matrix solutions and the matrix Weyl function  $W$} \label{sec:4} 

Similarly to the scalar Jacobi case,  Weyl coefficient (being  a matrix for  the block case), seems to be a natural candidate to be  the main object in any possible  method of  subordinacy, linking  generalized eigenvectors with  the absolutely continuous and with the singular part of the spectral measure. 
Observe first, that if $J$ is s.a., and $z \in \CC \setminus \RR$, then $z \notin \si(J)$, so denote:
\beql{ujz}
    u^{(j)}(z) := (J - z\Id)^{-1} \de_0(e_j), \ \ j=1,\ldots, d.
\eeq
In particular, for any $j$ we have $u^{(j)}(z) \in \Dom(J) \subset \ldnz{\CC^d}$ and 
\[
    \calJ u^{(j)}(z) = z u^{(j)}(z) + \de_0(e_j), \ \ j=1,\ldots, d.
\]
Considering the  terms $n \geq 1$ of the above equality of sequences we see that  each $u^{(j)}(z)$ is a gev for $J$ and $z$. Moreover, taking the term $n=0$ we get 
\beql{eq-for0}
    e_j = \left( (\calJ - z\Id) u^{(j)}(z) \right)_0 = 
    (B_0 - z \Id) u^{(j)}_0(z)+ A_0 u^{(j)}_1(z).
\eeq
Defining the matrix sequence with $u^{(j)}(z)$ as the sequence of its $j$th column, i.e., 
\beql{def-U(z)}
    \tilde{U}(z) := 
    [u^{(1)}(z), \ldots, u^{(d)}(z)] \in\ldnz{\Mdc}, 
\eeq
we obtain an mgev for $J$ and $z$. Now let $U(z) \in \ldnmj{\Mdc}$ be the  extension of $\tilde{U}(z)$ to an emgev for $J$ and $z$ (see,  e.g.,  \cite{Weyl} for some more  details of such extensions).
\begin{definition}
  We call $U(z)$ \  {\em the Weyl matrix  solution for $J$ and $z$}.  
\end{definition}
One can prove (see Proposition 5.3 in \cite{Weyl}) the following result being the generalization of the classical case $d=1$.

\begin{proposition}\label{prop-ltwosol-upr}
Let  $J$ be s.a. and $z \in \CC \setminus \RR$. There exists exactly one $\Wej(z)\in\Mdc$ such that  
\beql{eq-ltwo}
    P(z) \Wej(z) + Q(z) \in \ldnmj{\Mdc}.
\eeq
The above  unique $\Wej(z)$ satisfies: \    
$P(z) \Wej(z) + Q(z)$ \  is the Weyl  matrix  solution for $J$ and $z$, i.e., \  $ P(z) \Wej(z) + Q(z) = U(z)$; \  
$\Wej(z) = U_0(z)$ \ \ and \ 
 $\det \Wej(z) \neq 0$.
\end{proposition}

\begin{definition}\label{maWe}
Let $J$ be s.a. For fixed $z \in \CC \setminus \RR$\ 
the unique   $\Wej(z)$,  such that \eqref{eq-ltwo} holds is called {\em the matrix Weyl coefficient  (for $J$ and $z$)}. 
And the appropriate function 
$W : \CC \setminus \RR \strz \Mdc$ is called {\em the matrix Weyl function (for $J$)}.
\end{definition}
We omit here the dependence on $J$ in the notation, assuming that we consider a fixed $J$.

\subsection{The Cauchy transform of the spectral matrix measure and the matrix Weyl function}

Assume also here that  $J$ is s.a. \ 
The Cauchy transform  
of 
$M:=E_{J,\vec{\varphi}}$
(see \eqref{eq:IV:2.1}) is  defined as 
\beql{Cau-M}
  \Caj:\CC \setminus\RR\strz\Mdc, \qquad  \Caj(z): = \int_{\RR} \frac{1}{\la - z} \ud M(\la), \quad z \in \CC \setminus\RR,
\eeq
where this means that \ \ 
$\left(\Caj(z)\right)_{i,j}=
 \int_{\RR} \frac{1}{\la - z} \ud M_{i,j}(\la)$, 
with
$M_{i,j} = E_{J, \varphi_j, \varphi_i}$ \
and  $\varphi_j$  given by  \eqref{fix-fincyc}  for  $i,j=1,\ldots,d
$.
Hence, by spectral calculus for s.a. operators and by \eqref{ujz}, \eqref{l2scal} and \eqref{def-U(z)},   for \  $i,j=1,\ldots,d$ and $z \in \CC \setminus \RR$
\begin{align*}
    \left(\Caj(z)\right)_{i,j} 
    &= 
    \int_{\RR} \frac{1}{\la - z}\ud E_{J, \varphi_j, \varphi_i}  (\la)
= 
        \big\langle (J - z \Id)^{-1} \delta_0(e_j), \delta_0(e_i) \big\rangle_{\el^2}       \\
 &= 
    \big\langle u^{(j)}(z), \delta_0(e_i) \big\rangle_{\el^2} = 
    \big\langle (u^{(j)}(z))_0, e_i \big\rangle_{\CC^d} = 
    \left(U_0(z)\right)_{i,j},      \end{align*}
where  $U(z)$ is  Weyl matrix  solution for $J$ and $z$.
Finally, by Proposition~\ref{prop-ltwosol-upr}, we get
\begin{fact}\label{fact-WC}
If $J$ is  s.a.,   \ $z \in \CC \setminus \RR$, then \   \ 
$
\Wej(z) = \Caj(z) = 
    U_0(z) = 
    \Big(
        \big\langle (J - z \Id)^{-1} \delta_0(e_j), \delta_0(e_i) \big\rangle_{\el^2}         
    \Big)_{i,j=1,\ldots,d.}
    $
\end{fact}

\subsection{The boundary limits of the matrix Weyl function and spectral consequences for $J$} \label{sec:boundaryWeyl}

Assume  the s.a. of $J$, as before, and let us use the notation from the previous subsection, including  \  $M := E_{J,\vec{\varphi}}$.
Fact~\ref{fact-WC} \ implies that $W$  is  a holomorphic matrix-valued function. 
Moreover, $\Im W$ is a strictly positive matrix on $\CC_+$ by  \cite[Theorem 5.6]{Weyl}.  Thus, we see that the restriction of $W$ to $\CC_+$ is a matrix Herglotz function. Denote
\[
    \BLW := 
    \Big\{
        \la \in \RR: \lim_{\epsilon \to 0^+} \Wej(\la + i\epsilon) \  \mbox{exists}\footnote{As the  limit in $\Mdc$, i.e., in particular the limit must belong to $\Mdc$.}
    \Big\},
\]
and
\[
    \Wej(\la + i0) := 
    \lim_{\epsilon \to 0^+} \Wej(\la + i\epsilon), \quad 
    \la \in \BLW,
\]
and define the following sets:
\begin{align}
    \nonumber
    \Sac := 
    \bigcup_{r=1}^d \Sacr,\ \ \ \textrm{where}\ \ 
    \Sacr &:= 
    \big\{ \la \in \BLW : \rank \big( \Im \Wej(\la + i0) \big) = r \big\}, 
    \quad 1 \leq r \leq d, \\
    \label{Ssi}
    \Ssing &:= 
    \Big\{ 
        \la \in \RR : 
        \lim_{\epsilon \to 0^+} \Im \big( \tr \Wej(\la+i\epsilon) \big) = +\infty 
    \Big\}.
\end{align}
In particular it follows that  
\ $\Ssing \subset \RR \setminus \BLW$ \  and   \ \ 
$
    \Sac = \big\{ \la \in \BLW : \Im \Wej(\la + i0) \neq 0 \big\}$.

Let us recall now the crucial  result, joining   the above defined sets with  properties of \ $M_{\mathrm{ac}}$ \ and  \ $M_{\mathrm{sing}}$ \  --- the a.c. and the sing. parts  \ 
of the spectral matrix measure $M$ of $J$ w.r.t. \  $|\cdot|$ \   (see \cite[Fact A.6]{Weyl}).
This theorem is obtained just by the direct use of the abstract result 
\cite[Theorem 6.1]{Gesztesy2000} to the spectral  matrix measure $M$.
Define  $D: \BLW \strz\Mdc$  by 
\beql{eq:Mac:a}
    D(\la) := \frac{1}{\pi} \Im \Wej(\la + i0), \qquad \la \in \BLW. 
\eeq

\begin{theorem}\label{ThdensM}
\hspace{0.1 ex}
\begin{enumerate}[label=(\roman*)]
    \item \label{ThdensM:1}
    $\Ssing$  is a support of $\Msing$;
    
    \item \label{ThdensM:2}
    $|\RR\setminus \BLW|=0$;
    
    \item \label{ThdensM:3}
    $D$ is a density of $\Mac$ on $\BLW$ w.r.t. \ $|\cdot|$. 
\end{enumerate}
\end{theorem}

So, this result shows that controlling the boundary limits of the matrix Weyl function allows to get a lot of detailed information about the spectral matrix measure of $J$ and  of its sing. and a.c. parts. 
Theorem~\ref{ThdensM} can be ``translated'' into the   operator spectral language in the form as follows (see \cite{Weyl} subsection 5.4 for the proof).

\begin{proposition}
\label{cor:1}
Suppose that $J$ is s.a. and $G \in \Bor(\RR)$.
\begin{enumerate}[label=(\roman*)]
    \item \label{cor:1:1}
    If $G \subset \RR \setminus \Ssing$, then $J$ is absolutely continuous in $G$.
    
    \item \label{cor:1:2}
    If $G \subset \Sac\cup(\RR\setminus(\BLW\cup\Ssing))$, then $J$ is absolutely continuous in $G$ and \ $\clleb{G} \subset \sigmaAC(J)$. So, if moreover $G$ is open, or if $G$ is a sum of an arbitrary family of connected non-singletons in $\RR$, then  $\cl{G} \subset \sigmaAC(J)$.
\end{enumerate}
\end{proposition}

\subsection{Some analogs of ``J--L continuous interpolation''  
of a discrete  family of semi-norms} \label{J-Lsemi}

The main idea of  the Jitomirskaya--Last's new approach in   \cite{Jitomirskaya1999}  was ``technically'' based  on a continuous interpolation
of the discrete  family of semi-norms $\{\norm{\cdot}_{n}\}_{n \in \Nz}$ in $\lnz{\CC}$ to the family  $\{\norm{\cdot}_{t}\}_{t \in [0, +\infty)}$. Here we extend this notion in two ways. Let $V$ be a normed space,  $n_0 \in \ZZ$ and   $X \in \ell(\Nk{ n_0}, V)$.
For any $(n_1,t) \in \ZZ \times \RR$ such that $n_0 \leq n_1 \leq t$, we define 
\begin{equation} \label{eq:105a}
    \norm{X}_{[n_1,t]}:= 
    \bigg( 
    \sum_{k=n_1}^{\lfloor t \rfloor} 
        \norm{X_k}^2_V + \{t\} \! \norm{X_{\lfloor t \rfloor + 1}}^2_V 
    \bigg)^{1/2},
\end{equation}
where $\lfloor t \rfloor$ and $\{t\}$ are the integer  and the fractional part of $t$, respectively, and for $t=\infty$
\begin{equation} \label{eq:105b}
    \norm{X}_{[n_1,\infty]} := 
    \bigg( 
    \sum_{k=n_1}^{+\infty} \norm{X_k}^2_V
    \bigg)^{1/2}
\end{equation}
(which can  possibly be $+\infty$).
Moreover, if $V= \Mdc$ for some $d \geq 1$, then similarly to \eqref{eq:105a} we denote 
\begin{equation} \label{eq:105c}
    \lnorm{X}_{[n_1,t]}:= 
    \bigg( 
    \sum_{k=n_1}^{\lfloor t \rfloor} 
        \lnorm{X_k}^2 + \{t\} \!\! \lnorm{X_{\lfloor t \rfloor + 1}}^2 
    \bigg)^{1/2}.
\end{equation}

\section{A barrier approach to uniform estimates of matrix Weyl function} \label{sec:5}

This is the main part of the article. It contains the  main new ideas allowing to get some analogies of  the 1-dimensional subordinacy results also in the  $d$-block case. The most important new idea seems to be the usage of barriers, which allow us to control generalized eigenvectors uniformly with respect to initial conditions. Moreover, one of technical tools used  in our proof of Theorem~\ref{thm:9} is a generalisation of the Jitomirskaya--Last's idea from \cite{Jitomirskaya1999}.

\subsection{The nonsubordinacy} \label{sec:m-v-nonsubodinacy'}

The first notion worth mentioning before we focus on  the barrier nonsubordinacy is the simpler notion of nonsubordinacy, defined in  Introduction (see Definition~\ref{denonsub}). 

Note that in \eqref{eq:130'}  we could choose ``liminf over $t>0$'', instead of the more  original Khan and Pearson's like ``liminf over $n\in\NN$'',  and such choice would be equivalent. See \cite[subsection 6.1]{Weyl}  for the proof of this fact and also for  the remaining proofs for  the present small subsection,  and for other details related to the nonsubordinacy notion. 
Let us mention here also  that the similar notion of the so called {\em vector nonsubordinacy} can be defined as well just by considering  egev-s instead of emgev-s and with some  similar seminorms in the definition. And, surprisingly, this is   equivalent to the nonsubordinacy. The  new spectral result below is ''not so far from the absolute continuity of $J$''.

\begin{theorem} \label{prop:16}
Suppose that  $J$ is s.a. and it satisfies nonsubordinacy condition for some $\la \in \RR$. Then \ $\dim (\GEVld{\la}) = 0$ \ and \ $\la \in \sigma(J) \setminus \sigmaP(J)$.
\end{theorem}

\subsection{Barriers and barrier nonsubordinacy}

In what follows we need to make the concept of  nonsubordinacy more controllable than in Definition \ref{denonsub}. More precisely, our  goal is to  control two things:
\begin{itemize}
    \item  a bound on the ratio \ $\frac{\norm{U}_{[0,t]}}{\norm{V}_{[0,t]}}$ \  for any fixed ``large'' $t$ and  $\la\in G$, but joint for all such $U,V \in \MGEV{ \la}$ which are ``normalized'' in a  proper sense  --- it is a  key requirement;
    \item the size of the above bound as a function of $t$. 
\end{itemize}
This explains the sense of the notions of barrier and of barrier nonsubordinacy  defined  in Introduction --- see   Definition~\ref{def:barrier}. 
By the homogeneity of all the norms and semi-norms, and  by the use of the ``symmetry w.r.t. $U$ and $V$'' we easily obtain:
\begin{remark} \label{rem:1}
Let $\emptyset \neq G \subset \RR$ and $\barrier : G \times [1, \infty) \strz \RR$ be a barrier. Then
\begin{enumerate}[label=(\roman*)]
    \item \label{eq:107':1}
    $\barrier(\la, t) \geq 1$ \ for any $\la \in G$, $t \geq 1$.
        
    \item \label{eq:107':c}
    for any $\lambda \in G$, any non-zero $U,V \in \MGEV{\lambda}$ and any $t \geq 1$
    \[ 
        \frac{1}{\barrier(\la, t)} 
        \frac{\norm{U_{-1}}^2 + \norm{U_{0}}^2}{\norm{V_{-1}}^2 + \norm{V_{0}}^2} 
        \leq
        \frac{\norm{U}^2_{[0,t]}}{\norm{V}^2_{[0,t]}} 
        \leq
        \barrier(\la, t)
        \frac{\norm{U_{-1}}^2 + \norm{U_{0}}^2}{\norm{V_{-1}}^2 + \norm{V_{0}}^2}.
    \]
\end{enumerate}
\end{remark}
Hence, using ``liminf over $t>0$'' for nonsubordinacy (see subsection \ref{sec:m-v-nonsubodinacy'}), we get the expected result.   
\begin{remark}
\label{barvecmatrnonsub}
 If $G \subset \RR$,  $\barrier$ is  a barrier on $G$ for $J$ and   {$J$ is $\barrier$-nonsubordinate} on $G$, then $J$ satisfies  nonsubordinacy condition  for any $\la\in G$.  
\end{remark}

A natural problem now is the existence of barriers for fixed $J$ and $G$. And  ''an abstract'' solution of it is simple: 
 for any $J$ there  exists even its  smallest ``universal'' barrier (on $\RR$, so ''on each $G$'') --- we call it 
 {\em the minimal barrier for $J$}.
 For  $\la\in \CC$  \     denote
 $\NMGEV{\la}:= \{ U \in \MGEV{\la}: \ (U_{-1}, U_0) \in \sS \}$, where \ 
 $\sS = 
 \big\{ (X, X') \in \Mdc \times \Mdc: \ \norm{X}^2 + \norm{X'}^2 = 1 \big\}
$, \ 
and consider the  function \ \  $\barrierMIN : \RR \times [1, +\infty) \to [1, +\infty]$ \ \  given by the formula \beql{barmin}
     \barrierMIN(\la, t) :=
    \sup_{U,V \in \NMGEV{\la}}
     \left(\frac{\norm{U}_{[0,t]}}{\norm{V}_{[0,t]}}\right)^2, \qquad \la\in\RR, \  t\geq 1.
 \eeq
Now, using the compactness of $\sS \times \sS$ and the standard  Weierstrass theorem arguments, we easily get:
 \begin{proposition} \label{prop:19}
\ \   $\barrierMIN$ has  finite values and   
 is a barrier for $J$ on $\RR$. Moreover $\restr{\barrierMIN}{G\times [1, +\infty)}$ is the smallest barrier for $J$ on any $G \subset \RR$, i.e., $\barrierMIN(\la,t)\leq\barrier(\la,t)$  for any barrier $\barrier$ for $J$ on $G$, 
$\la\in G, \  t\geq 1$.
 \end{proposition}

Our next goal is now the construction of  a ''more computable''  barrier. 
We shall do this in terms of the sequence of  $n$-step transfer matrices, i.e.,  
$$
R(\la)=    (R_{n}(\la))_{n\in\N}.
$$

To do this, we need  first:

\begin{proposition}
Let $\la \in \RR$.
If $u \in \GEV{ \la}$, then for any $t \geq 1$
\begin{equation} \label{eq:87}
    \frac{1}{2}
    \big( \norm{u_{-1}}_{\CC^d}^2 +\norm{u_0}_{\CC^d}^2 \big)
    \lnorm{R(\la)}_{[1,t]}^2
    \leq
    \norm{u}_{[0,t]}^2 
    \leq
    \big( \norm{u_{-1}}_{\CC^d}^2 +\norm{u_0}_{\CC^d}^2 \big)
    \norm{R(\la)}_{[1,t]}^2.
\end{equation}
If $U \in \MGEV{ \la}$, then for any $t \geq 1$
\begin{equation} \label{eq:88}
    \frac{1}{4}
    \big( \norm{U_{-1}}^2 +\norm{U_0}^2 \big)
    \lnorm{R(\la)}_{[1,t]}^2
    \leq
    \norm{U}_{[0,t]}^2 
    \leq
    2
    \big( \norm{U_{-1}}^2 +\norm{U_0}^2 \big)
    \norm{R(\la)}_{[1,t]}^2.
\end{equation}
In particular, for any $U,V \in \NMGEV{\la}$
\begin{equation} \label{eq:88'}
\left(\frac{\norm{U}_{[0,t]}}{\norm{V}_{[0,t]}}\right)^2 
\leq 8
\left(\frac{\norm{R(\la)}_{[1,t]}}
{\lnorm{R(\la)}_{[1,t]} }\right)^2, 
\qquad t\geq 1.
\end{equation}
\end{proposition}
\begin{proof}
If $u_{-1} = u_0 = 0$, then $u = 0$, and the first assertion holds, so  assume that $\norm{u_{-1}}_{\CC^d}^2 + \norm{u_0}_{\CC^d}^2 > 0$. Set
\[
    \vec{u}_k :=
    \begin{pmatrix}
        u_{k-1} \\
        u_k
    \end{pmatrix}, \quad k \geq 0.
\]
Then for any $k \geq 1$ it holds
$
    \vec{u}_k
    =
    R_k(\la)
    \vec{u}_0
$.
Thus
\begin{equation} \label{eq:89}
    \norm{\vec{u}_0}_{\CC^{2d}}^2 \lnorm{R_k(\la)}^2
    \leq 
    \norm{\vec{u}_k}_{\CC^{2d}}^2
    \leq 
    \norm{\vec{u}_0}_{\CC^{2d}}^2 \norm{R_k(\la)}^2, \quad k \geq 1.
\end{equation}
Then by \cite[Observation 2.4]{Weyl} we get
\[
    \norm{\vec{u}_0}_{\CC^{2d}}^2
    \lnorm{R(\la)}^2_{[1,t]}
    \leq 
    \norm{\vec{u}}_{[1,t]}^2
    \leq 
    \norm{\vec{u}_0}_{\CC^{2d}}^2
    \norm{R(\la)}^2_{[1,t]}, \quad t \geq 1.
\]
Thus \eqref{eq:87} follows, since we have
\[
    \norm{u}_{[0,t]}^2
    \leq
    \norm{\vec{u}}_{[1,t]}^2
    \leq 
    2 \norm{u}_{[0,t]}^2
\]
 by direct computations.
Now, consider $U \in \MGEV{ \la}$. Then by \cite[Fact 4.4]{Weyl} \ for any $v \in \CC^d$ the sequence $U v$ is a vector generalized eigenvector. Thus \eqref{eq:89} implies
\begin{equation} \label{eq:90}
    \big( \norm{U_{-1} v}_{\CC^{d}}^2 + \norm{U_0 v}_{\CC^{d}}^2 \big) \lnorm{R_k(\la)}^2
    \leq 
    \norm{
    \begin{pmatrix}
        U_{k-1} v\\
        U_k v
    \end{pmatrix}
    }_{\CC^{2d}}^2
    \leq 
    \big( \norm{U_{-1} v}_{\CC^{d}}^2 + \norm{U_0 v}_{\CC^{d}}^2 \big) \norm{R_k(\la)}^2.
\end{equation}
In particular,
\[
    \big( \norm{U_{-1} v}_{\CC^{d}}^2 + \norm{U_0 v}_{\CC^{d}}^2 \big) \lnorm{R_k(\la)}^2
    \leq 
    \big( \norm{U_{k-1}}^2 + \norm{U_k}^2 \big) \norm{v}_{\CC^{d}}^2.
\]
Thus
\begin{equation} \label{eq:91}
    \frac{1}{2} 
    \big( \norm{U_{-1}}^2 + \norm{U_0}^2 \big) \lnorm{R_k(\la)}^2
    \leq 
    \norm{U_{k-1}}^2 + \norm{U_k}^2.
\end{equation}
On the other hand, by \eqref{eq:90}
\[
    \norm{U_{k-1} v}_{\CC^{d}}^2 + \norm{U_k v}_{\CC^{d}}^2 
    \leq
    \big( \norm{U_{-1}}^2 + \norm{U_0}^2 \big) \norm{v}_{\CC^{d}}^2 \norm{R_k(\la)}^2, 
\]
and consequently,
\begin{equation} \label{eq:92}
    \norm{U_{k-1}}^2 + \norm{U_k}^2 
    \leq
    2 \big( \norm{U_{-1}}^2 + \norm{U_0}^2 \big) \norm{R_k(\la)}^2.
\end{equation}
For any $k \geq 0$ let us define 
\[
    \vec{x}_k :=
    \begin{pmatrix}
        \norm{U_{k-1}} \\
        \norm{U_k}
    \end{pmatrix} \in \CC^2.
\]
Therefore, by combining \eqref{eq:91} and \eqref{eq:92} we obtain
\[
    \frac{1}{2} 
    \norm{\vec{x}_0}_{\CC^2}^2 \lnorm{R_k(\la)}^2
    \leq 
    \norm{\vec{x}_k}_{\CC^2}^2
    \leq 
    2 \norm{\vec{x}_0}_{\CC^2}^2 \norm{R_k(\la)}^2, \quad k \geq 1,
\]
which is an analogue of \eqref{eq:89}.
Now  we get  \eqref{eq:88} by a similar manner.
\end{proof}

In view of this result let us consider now \ \ 
$\barrierTR : \RR \times [1, +\infty) \to \RR$ \  given  by the formula

\begin{equation} \label{eq:120}
    \barrierTR(\la, t) := 
    8 \left(\frac{\norm{R(\la)}_{[1,t]}}{\lnorm{R(\la)}_{[1,t]}}\right)^2, \qquad   \la\in\RR, \  t\geq 1.  
\end{equation}

The last assertion of the proposition yields exactly:

\begin{corollary} \label{cor:4}
  $\barrierTR$ \  is  a barrier for $J$ on $\RR$.
\end{corollary}

For any $G\subset\RR$ \ \  we call \ $\restr{\barrierTR}{G\times[1,+\infty)}$ \ \ \  {\em the transfer matrix barrier for $J$ on $G$}.

The convenience and the importance of our definitions of barrier and of the barrier-nonsubordinacy will be  visible in  proofs in  the next subsections.
 
\subsection{The main result and its consequences}

Assume, as before, that $J$ is s.a. and consider the matrix Weyl function $\Wej$ for $J$ (see Definition~\ref{maWe}). Before we start to ``control'' in a sense its boundary limits  we need to make  use of the choice of Jitomirskaya--Last type semi-norms, and formulate a result being a block case analog of the appropriate   scalar case one result from \cite{Jitomirskaya1999}.

\begin{proposition} \label{defJL}
Suppose that $J$ is s.a. For any  $\la \in \RR$  there exists a unique   \ $\JLf{\la} : \RR_+ \to \RR_+$ \  satisfying 
\begin{equation} \label{eq:13}
	\norm{P(\la)}_{[0, \JLf{\la}(\epsilon)]} 
	\norm{Q(\la)}_{[0, \JLf{\la}(\epsilon)]} = \frac{1}{2 \epsilon},\quad \eps>0.
\end{equation}
 Moreover, $\JLf{\la}$ is a strictly decreasing continuous function and satisfies
\begin{equation} \label{eq:106a}
    \lim_{\epsilon \to 0^+} \JLf{\la}(\epsilon) = +\infty, \quad
    \lim_{\epsilon \to +\infty} \JLf{\la}(\epsilon) = 0.
\end{equation}
Consequently, its inverse $\JLf{\la}^{-1}$ is also a strictly decreasing continuous function and satisfies
\begin{equation} \label{eq:106b}
    \lim_{t \to 0^+} \JLf{\la}^{-1}(t) = +\infty, \quad
    \lim_{t \to +\infty} \JLf{\la}^{-1}(t) = 0.
\end{equation}
\end{proposition}

The proof is  based on the fact that the  function 'of $t$' from $\RR_+$ to $\RR_+$ given by \ $
    \norm{P(\la)}_{[0,t]} 
    \norm{Q(\la)}_{[0,t]}$ \ is continuous, strictly increasing and surjective (see  \cite[Proposition 4.6 and its proof]{Weyl}).
And  $\JLf{\la}$ is called \emph{J-L function (for $J$ and $\la$)}.

We are ready to formulate our main result ``on controlling  the boundary limits of the matrix Weyl function'', being probably the most important result of this work.
\begin{theorem} \label{thm:9}
Assume that $J$ is s.a. and $G\subset\RR$. If $\barrier$ is a barrier for $J$ on $G$, then for any $\la \in G$ and any  $\eps>0$ with $\JLf{\la}(\epsilon) \geq 1$ 
\begin{equation} \label{eq:116}
    \left(8 \barrier \big( \la, \JLf{\la}(\epsilon) \big)\right)^{-1}\Id
    \leq 
    \Im \Wej(\la + i \epsilon)
\end{equation}
and
\begin{equation} \label{eq:117}
    s_-(\la, \epsilon) \leq
    \norm{\Wej(\la + i \epsilon)} \leq 
    s_+(\la, \epsilon), 
\end{equation}
where
\begin{equation} \label{eq:115}
    s_{\pm}(\la, \epsilon) := 
    4d \barrier \big( \la, \JLf{\la}(\epsilon) \big) \pm \sqrt{ \Big( 4d \barrier \big( \la, \JLf{\la}(\epsilon) \big) \Big)^2 - 1}.
\end{equation}
\end{theorem}

We present the proof in the next subsection. But now let us remark  only that the square root in \eqref{eq:115} is at least $15$, since $d \geq 1$ and $\barrier(\la, t) \geq 1$ for any $t \geq 1$. Consequently, both $s_-(\la, \epsilon)$ and $s_+(\la, \epsilon)$ are positive for the considered $\la$ and $\epsilon$, hence both estimates in \eqref{eq:117} could be of significant importance.

Having Theorem~\ref{thm:9} and all the delicate relations between various  objects connected somehow to $J$ and  described in the previous sections, we can finally  formulate and prove the most important abstract spectral result of this article.

Denote for short the spectral  matrix measure for $J$  by $M$, i.e., 
$
    M := E_{J,\vec{\varphi}},
$
where  $\vec{\varphi}=(\varphi_1, \ldots, \varphi_k)$ is canonical cyclic system \eqref{fix-fincyc} for $J$. Recall that some  important results and some notation related to the absolutely continuous and the singular part  of  $M$, such as Theorem~\ref{ThdensM}, the set $\BLW$ ``with boundary limits for $\Wej$'' and the density $D$ (see \eqref{eq:Mac:a}) of $\Mac$ on $\BLW$ w.r.t. the Lebesgue measure \ $|\cdot|$ are  presented in  Section~\ref{sec:boundaryWeyl}. Recall that  the absolute continuity is understood with respect to the Lebesgue measure, if no other choice is made.

\begin{theorem} \label{thm:10}
Assume that $J$ is s.a.,   $G \in \Bor(\RR)$ and $\barrier$ is a barrier for $J$ on $G$. If $J$ is $\barrier$-nonsubordinate on $G$, then
\begin{enumerate}[label=(\alph*)]
    \item \label{thm:10:1}
    $M$ is absolutely continuous on $G$.
    
    \item \label{thm:10:2}
    the density $D$ of $\Mac$ is  an invertible matrix at any  $\la\in G \cap \BLW$. 
    
    \item \label{thm:10:3}
    $J$ is absolutely continuous in $G$ and \ \ $\clleb{G} \subset \sigmaAC(J)$. 
\end{enumerate}
If, moreover, $J$ is uniformly $\barrier$-nonsubordinate on $G$, then there exist  $c_1, c_2>0$ such that
\beql{szacDens}
    c_1 \Id \leq D(\la) \leq c_2 \Id, \quad \la \in G \cap \BLW.
\eeq
\end{theorem}

\begin{proof}
Suppose that $J$ is $\barrier$-nonsubordinate on $G$.
By Theorem~\ref{thm:9} 
\begin{equation} \label{eq:110}
    \liminf_{\epsilon \to 0^+} \norm{\Wej(\la + i \epsilon)} \leq
    8d \liminf_{\epsilon \to 0^+} \barrier \big( \la, \JLf{\la}(\epsilon) \big), \quad \la \in G.
\end{equation}
By the definition of ``$\liminf$'' there  exists a 
sequence $(t_k)_{k \in \NN}$  in $[1, +\infty)$ with \ 
$\lim_{k \to +\infty} t_k = +\infty$, such that 
\[
    \liminf_{t \to +\infty} 
    \barrier(\la, t)
    =
    \lim_{k \to +\infty} 
    \barrier(\la, t_k).
\]
By Proposition~\ref{defJL} for each $k$ we  define
\[
    \epsilon_k := \JLf{\la}^{-1}(t_k)
\]
and by \eqref{eq:106b} we get
\[
    \lim_{k \to +\infty} \epsilon_k = 0.
\]
Hence, by  \eqref{eq:132} 
\begin{equation} \label{eq:118}
    \liminf_{\epsilon \to 0^+} 
    \barrier \big( \la, \JLf{\la}(\epsilon) \big)
    \leq
    \lim_{k \to +\infty} 
    \barrier \big( \la, \JLf{\la}(\epsilon_k) \big)
    =
    \lim_{k \to +\infty} \barrier(\la, t_k)
    = 
    \liminf_{t \to +\infty} \barrier(\la, t)<+\infty,
\end{equation}
which by \eqref{eq:110}  implies 
\begin{equation} \label{eq:108}
    \liminf_{\epsilon \to 0^+} \norm{\Wej(\la + i \epsilon)} < +\infty, \quad \la \in G,
\end{equation}
and by \eqref{Ssi} this, in particular, yields
 $G \subset\ \RR \setminus \Ssing$. Now, using Theorem~\ref{ThdensM}\ref{ThdensM:1}, we obtain assertion~\ref{thm:10:1}:    $M$ is a.c. on $G$.

Observe now that \eqref{eq:116} implies
\[
    \frac{1}{8 \liminf_{\epsilon \to 0^+} \barrier \big( \la, \JLf{\la}(\epsilon) \big)} \norm{v}_{\CC^{d}}^2 
    \leq 
    \limsup_{\epsilon \to 0^+}
    \big\langle 
        \big( \Im \Wej(\la + i \epsilon) \big) v, v
    \big\rangle_{\CC^{d}},\qquad v\in\CC^d,
\]
which by \eqref{eq:118} and $\barrier$-nonsubordinacy yields 
\begin{equation}\label{forma}
    c(\la) \norm{v}_{\CC^{d}}^2 
    \leq 
    \limsup_{\epsilon \to 0^+}
    \big\langle 
        \big( \Im \Wej(\la + i \epsilon) \big) v, v
    \big\rangle_{\CC^{d}}, \qquad v\in\CC^d,
\end{equation}
where
\begin{equation}\label{forma'}
    c(\la) = \Big( 8 \liminf_{t \to +\infty} \barrier(\la, t) \Big)^{-1} > 0, \quad \la \in G.
\end{equation}

Assume now that $\la \in G \cap \BLW$.
To get assertion~\ref{thm:10:2} \   in view of \eqref{eq:Mac:a} \  it is enough to prove that 
 the matrix $\Im \big( \Wej(\la + i 0) \big)$ is invertible.
But now   
\[
    \limsup_{\epsilon \to 0^+} \Im \Wej(\la + i \epsilon)=\lim_{\epsilon \to 0^+} \Im \Wej(\la + i \epsilon)=\Wej(\la + i0),
\]     
so,     by \eqref{forma},
\begin{equation}\label{forma-dol}
    c(\la)
    \leq 
     \Im \Wej(\la + i0) 
\end{equation}
Therefore  $\Im \Wej(\la + i 0)$ is invertible, by \eqref{forma'}, and   $G \cap \BLW \subset \Sacd\subset\Sac$.

Now Proposition~\ref{cor:1}\ref{cor:1:2} gives
\[
\clleb{G \cap \BLW} \subset \sigmaAC(J).
\]
But $\clleb{G \cap \BLW} = \clleb{G}$, and  the conclusion~\ref{thm:10:3} follows.

In the  uniform case, the LHS bound on $D(\la)$ follows directly  from \eqref{forma-dol} and \eqref{forma'}. And to get the RHS estimate, it  suffices to use the fact that the norm of a matrix gives the upper bound for the  quadratic form, and so it is  sufficient to use \eqref{eq:110}.
\end{proof}

\subsection{The proof of the main result}

Before we turn to the proof of Theorem~\ref{thm:9} we need some preparations.
For fixed $\la \in \RR$ and $\epsilon > 0$ let us define
\begin{align}
	\label{eq:14a}
	U_{\epsilon} &:= Q(\la + i \epsilon)  + P(\la + i \epsilon) \Wej(\la + i \epsilon) \\
	\label{eq:14b}
	V_\epsilon &:= Q(\la) + P(\la) \Wej(\la + i \epsilon).
\end{align}

The following lemma can be found  e.g. in \cite[Section 2]{Kostyuchenko1998} and \cite[Proposition 4.8 with the complete proof]{Weyl}.
\begin{lemma}
Let $F \in \lnz{\Mdc}$ and $z\in\CC$. The unique solution $(S_n)_{n\geq-1}\in \lnmj{\Mdc}$ \ of
\begin{equation} \label{eq:59}
    A_n S_{n+1} + B_n S_n + A_{n-1}^* S_{n-1} = z S_n + F_n, \quad n \geq 0.
\end{equation}
with the initial conditions $S_{-1} = S_0 = 0$ is equal to
\begin{equation} \label{eq:58}
    S_n := 
    \sum_{k=0}^{n-1} 
    \Big( Q_n(z) \big( P_k(\overline{z}) \big)^* - P_n(z) \big( Q_k(\overline{z}) \big)^* \Big) F_k, \quad n \geq -1.
\end{equation}
\end{lemma}

\begin{corollary}
For $\la \in \RR$ define an operator $L : \lnz{\Mdc} \to \lnz{\Mdc}$ by the formula
\begin{equation} \label{eq:61}
    (L F)_m =
    \sum_{k=0}^{m-1} \Big( Q_m(\la) \big( P_k(\la)\big)^* - P_m(\la) \big(Q_k(\la)\big)^* \Big) F_k, \quad m \geq 0.
\end{equation}
Then the sequences $U_\epsilon$ and $V_\epsilon$ defined in \eqref{eq:14a} and \eqref{eq:14b} satisfy for any $X \in \Mdc$
\begin{equation} \label{eq:77}
    (\Id - i \epsilon L) (U_\epsilon X) = V_\epsilon X.
\end{equation}
\end{corollary}
\begin{proof}
Take $X \in \Mdc$.
Observe that $U_\epsilon X$ satisfies \eqref{eq:59} with $z=\la$ and $F_n = i \epsilon (U_{\epsilon})_n X$. Similarly, $V_\epsilon X$ satisfies \eqref{eq:59} with $z = \la$ and $F_n \equiv 0$. Since the initial conditions of $V_\epsilon X$ and $U_\epsilon X$ are the same we get that the sequence $S := U_{\epsilon} X - V_\epsilon X$ can be expressed in the form \eqref{eq:58}. Thus, according to the definition \eqref{eq:61} we can write
\[
    U_\epsilon X - V_\epsilon X = i \epsilon L (U_\epsilon X),
\]
which is equivalent to \eqref{eq:77}.
\end{proof}

\begin{lemma}
Let the operator $L : \lnz{\Mdc} \to \lnz{\Mdc}$ be as above. 
Then for any $t \in [0, +\infty)$
\begin{equation}
    \label{eq:64}
    \norm{L F}_{[0,t]} \leq 2 \norm{P(\la)}_{[0,t]} \norm{Q(\la)}_{[0,t]} \norm{F}_{[0,t]}.
\end{equation}
\end{lemma}
\begin{proof}
Observe that by \eqref{eq:61} for any $m \in \Nz$ we have
\begin{align*}
    \norm{(L F)_m}
    &\leq 
    \norm{Q_m(\la)} \sum_{k=0}^{m-1} \norm{P_k(\la)} \norm{F_k} \\
    &+
    \norm{P_m(\la)} \sum_{k=0}^{m-1} \norm{Q_k(\la)} \norm{F_k}.
\end{align*}
Hence, by Cauchy--Schwarz inequality
\begin{equation} \label{eq:74}
    \norm{(L F)_m} 
    \leq 
    \Big(
    \norm{Q_m(\la)} \norm{P(\la)}_{[0,m-1]}
    +
    \norm{P_m(\la)} \norm{Q(\la)}_{[0,m-1]} 
    \Big)
    \norm{F}_{[0,m-1]}.
\end{equation}
Let $t \in [0, +\infty)$. Then $t \in [n, n+1)$ for some $n \in \Nz$. Therefore, by \eqref{eq:74}
\[
    \norm{(L F)_m} 
    \leq 
    v_m
    \norm{F}_{[0,t]}, \quad m=0,1,\ldots, n+1,
\]
where
\begin{equation} \label{eq:75}
    v_m := 
    \norm{Q_m(\la)} \norm{P(\la)}_{[0,t]}
    +
    \norm{P_m(\la)} \norm{Q(\la)}_{[0,t]}, \quad m=0,1,\ldots,n+1.
\end{equation}
Consequently,
\begin{align}
    \nonumber
    \norm{LF}_{[0,t]}
    &=
    \bigg(
    \sum_{m=0}^{n} \norm{(L F)_m}^2 + \{t\} \norm{(L F)_{n+1}}^2
    \bigg)^{1/2} \\
    \label{eq:76}
    &\leq 
    \bigg(
    \sum_{m=0}^{n} v_m^2 + \{t\} v_{n+1}^2
    \bigg)^{1/2} \norm{F}_{[0,t]}.
\end{align}
Now, by triangle inequality in $\CC^{n+2}$ with Euclidean norm and \eqref{eq:75} we get
\begin{align*}
    \bigg(
    \sum_{m=0}^{n} v_m^2 + \{t\} v_{n+1}^2
    \bigg)^{1/2} 
    &\leq 
    \norm{P(\la)}_{[0,t]}
    \bigg( 
    \sum_{m=0}^n \norm{Q_m(\la)}^2
    + 
    \{t\} \norm{Q_{n+1}(\la)}^2
    \bigg)^{1/2} \\
    &+
    \norm{Q(\la)}_{[0,t]}
    \bigg( 
    \sum_{m=0}^n \norm{P_m(\la)}^2
    + 
    \{t\} \norm{P_{n+1}(\la)}^2
    \bigg)^{1/2} \\
    &=
    2 \norm{P(\la)}_{[0,t]} \norm{Q(\la)}_{[0,t]}
\end{align*}
which combined with \eqref{eq:76} gives \eqref{eq:64}.
\end{proof}

\begin{proposition} \label{prop:8}
For any $X \in \Mdc$ and $t \in [0, +\infty)$ one has
\begin{equation} \label{eq:78}
	\norm{U_{\epsilon} X}_{[0,t]} \geq 
	\norm{V_\epsilon X}_{[0,t]} - 
	2 \epsilon  
	\norm{P(\la)}_{[0,t]} 
	\norm{Q(\la)}_{[0,t]} 
	\norm{U_{\epsilon} X}_{[0,t]}.
\end{equation}
Consequently, if $\JLf{\la}$ is J-L function, then
\begin{equation} \label{eq:79}
	2 \norm{U_{\epsilon} X}_{[0,\JLf{\la}(\epsilon)]} \geq 
	\norm{V_{\epsilon} X}_{[0,\JLf{\la}(\epsilon)]}.
\end{equation}
\end{proposition}
\begin{proof}
By \eqref{eq:77}
\[
    (\Id - i \epsilon L) (U_{\epsilon} X) = V_{\epsilon} X.
\]
Thus, by \eqref{eq:64} we get
\[
    \norm{V_{\epsilon} X}_{[0,t]}
    \leq 
    \big( 1 + 2 \epsilon \norm{P(\la)}_{[0,t]} \norm{Q(\la)}_{[0,t]} \big) 
    \norm{U_\epsilon X}_{[0,t]}
\]
which implies \eqref{eq:78}. Then the inequality \eqref{eq:79} follows from \eqref{eq:13}.
\end{proof}

\begin{proposition} \label{prop:9}
For any $v \in \CC^d$, $\la \in \RR$ and any $\epsilon > 0$ one has
\begin{equation} \label{eq:114a}
    \frac{1}{\epsilon}
    \big\langle \big( \Im \Wej(\la + i \epsilon) \big) v, v \big\rangle_{\CC^d}
    =
    \norm{ U_\epsilon v}_{[0, +\infty]}^2.
\end{equation}
Moreover,
\begin{equation} \label{eq:114b}
	\norm{U_{\epsilon}}^2_{[0,+\infty]} \leq
	\frac{\tr \big( \Im \Wej(\la+i \epsilon) \big)}{\epsilon}.
\end{equation}
\end{proposition}
\begin{proof}
We have $U_{\epsilon} = U(\la+i\eps)$ by \eqref{eq:14a}, so the result follows from \cite[Theorem 5.6]{Weyl}.
\end{proof}

We are ready to prove our main result.
\begin{proof}[Proof of Theorem~\ref{thm:9}]
By Proposition~\ref{prop:8} applied to $X=\Id$ and by \eqref{eq:114b} we have
\begin{align}
    \nonumber
	\norm{V_{\epsilon}}_{[0,\JLf{\la}(\epsilon)]} 
	&\leq 
	2 \norm{U_{\epsilon}}_{[0,\JLf{\la}(\epsilon)]} \\
	\nonumber
	&\leq
	2 \norm{U_{\epsilon}}_{[0,+\infty]} \\
	\label{eq:113}
	&\leq
	\frac{2}{\sqrt{\epsilon}} \sqrt{\tr \big( \Im \Wej(\la+i \epsilon) \big)}.
\end{align}
Now
\[
	\tr \big( \Im \Wej(\la+i \epsilon) \big)
	\leq
	d \norm{\Im \Wej(\la+i \epsilon)}
	\leq
	d \norm{\Wej(\la+i \epsilon)}
\]
(see, e.g., \cite[Proposition 2.2 {\em (ii), (v)}]{Weyl}). Thus,
\begin{equation} \label{eq:51}
	\norm{V_{\epsilon}}_{[0, \JLf{\la}(\epsilon)]}^2 \leq
	\frac{4 d}{\epsilon} \norm{W (\la + i \epsilon)}.
\end{equation}
On the other hand, by Remark~\ref{rem:1} and \eqref{eq:13}
\begin{align*}
    \norm{V_{\epsilon}}^2_{[0, \JLf{\la}(\epsilon)]}
    &= 
    \frac{\norm{V_{\epsilon}}_{[0, \JLf{\la}(\epsilon)]}}
    {\norm{P(\la)}_{[0, \JLf{\la}(\epsilon)]}} 
    \frac{\norm{V_{\epsilon}}_{[0, \JLf{\la}(\epsilon)]}} 
    {\norm{Q(\la)}_{[0, \JLf{\la}(\epsilon)]}}
    \norm{P(\la)}_{[0, \JLf{\la}(\epsilon)]}
    \norm{Q(\la)}_{[0, \JLf{\la}(\epsilon)]}
    \\
    &\geq 
    \frac{1}{\barrier(\la, \JLf{\la}(\epsilon))} 
    \big( 1 + \norm{\Wej(\la + i \epsilon)}^2 \big)
    \frac{1}{2 \epsilon}.
\end{align*}
Thus, by combining it with \eqref{eq:51} we obtain
\begin{equation} \label{eq:103}
    1 + \norm{\Wej(\la + i \epsilon)}^2 
    \leq 
    8d \barrier(\la, \JLf{\la}(\epsilon)) \norm{\Wej(\la + i \epsilon)}.
\end{equation}
This inequality is equivalent to
\[
    s^2 - 8d \barrier \big(\la, \JLf{\la}(\epsilon) \big) s + 1 \leq 0,
\]
where $s = \norm{\Wej(\la + i \epsilon)} \geq 0$. Its discriminant is equal to
\begin{equation} \label{eq:52}
    \Delta = \Big( 8d \barrier \big( \la, \JLf{\la}(\epsilon) \big) \Big)^2 - 4.
\end{equation}
Since $d \geq 1$ and $\barrier \big( \la, \JLf{\la}(\epsilon) \big) \geq 1$, we have
\begin{equation} \label{eq:53}
    \Delta \geq 8^2 - 4 = 60 > 0.
\end{equation}
Thus
\[
    s^2 - 8d \barrier \big( \la, \JLf{\la}(\epsilon) \big) s + 1 = 
    (s-s_-)(s-s_+) \leq 0,
\]
where
\[
    s_- = \frac{8d \barrier \big( \la, \JLf{\la}(\epsilon) \big) - \sqrt{\Delta}}{2} \quad \text{and} \quad
    s_+ = \frac{8d \barrier \big( \la, \JLf{\la}(\epsilon) \big) + \sqrt{\Delta}}{2}.
\]
By \eqref{eq:53} we see that $s_-, s_+ \in \RR$. Consequently, by \eqref{eq:52} we have $s_-, s_+ > 0$. Thus,
\[
    s_- \leq \norm{\Wej(\la + i \epsilon)} \leq s_+
\]
and \eqref{eq:117} follows.

The proof of \eqref{eq:116} is even simpler. Namely, let $v \in \CC^d$. Then by Proposition~\ref{prop:8} applied to $X=E^v$ (see \eqref{def:Ev}), \cite[Proposition 2.2  {\em(iv)}]{Weyl} and \eqref{eq:114a}
\[
    \norm{V_\epsilon v}_{[0, \JLf{\la}(\epsilon)]}^2 
    \leq 
    4 \norm{U_\epsilon v}_{[0, \JLf{\la}(\epsilon)]}^2
    =
    \frac{4}{\epsilon} 
    \big\langle \big( \Im \Wej(\la + i \epsilon) \big) v, v \big\rangle_{\CC^d}.
\]
Now again by \cite[Proposition 2.2  {\em(iv)}]{Weyl}, by  Remark~\ref{rem:1} and by \eqref{eq:13}
\begin{align*}
    \norm{V_\epsilon v}_{[0, \JLf{\la}(\epsilon)]}^2 
    &=
    \norm{V_\epsilon (E^v)}_{[0, \JLf{\la}(\epsilon)]}^2 \\
    &=
    \frac
    {\norm{V_\epsilon (E^v)}_{[0, \JLf{\la}(\epsilon)]}}
    {\norm{P(\la)}_{[0, \JLf{\la}(\epsilon)]}}
    \frac
    {\norm{V_\epsilon (E^v)}_{[0, \JLf{\la}(\epsilon)]}}
    {\norm{Q(\la)}_{[0, \JLf{\la}(\epsilon)]}}
    \norm{P(\la)}_{[0, \JLf{\la}(\epsilon)]} 
    \norm{Q(\la)}_{[0, \JLf{\la}(\epsilon)]} \\
    &\geq 
    \frac{1}{\barrier \big(\la, \JLf{\la}(\epsilon) \big)} \norm{v}_{\CC^d}^2 \frac{1}{2 \epsilon}.
\end{align*}
Therefore, 
\[
    \frac{1}{\barrier \big(\la, \JLf{\la}(\epsilon) \big)} \norm{v}_{\CC^d}^2 \leq 
    8 \big\langle \big( \Im \Wej(\la + i \epsilon) \big) v, v \big\rangle_{\CC^d}
\]
from which \eqref{eq:116} follows. The proof is complete.
\end{proof}

\section{Some sufficient conditions for nonsubordinacy} \label{sec:6}
In this section we extend some well-known conditions implying nonsubordinacy from $d=1$ to the general case.

\subsection{Generalized Last--Simon and Behncke--Stolz conditions}
\begin{definition}
Let $G \subset \RR$ be non-empty. We say that a s.a.  Jacobi matrix $J$ satisfies \emph{Generalized Last--Simon} (GLS in short) condition on $G$ if
\begin{equation} \label{eq:37b}
    \liminf_{n \to +\infty} 
    \frac{1}{\rho_n} 
    \sum_{k=1}^{n} \norm{R_k(\la)}^2 < +\infty, \quad \la \in G,
\end{equation}
where $R_k$ is $k$-step transfer matrix (see \eqref{eq:55}) and
\begin{equation} \label{eq:44}
     \rho_n := \sum_{k=0}^{n-1} \frac{1}{\norm{A_k}}.
\end{equation}
Moreover, if \eqref{eq:37b} is finite even after taking the supremum over $\la \in G$, we say that $J$ satisfies \emph{uniform GLS} condition on $G$.
\end{definition}

GLS condition has been introduced in \cite{LastSimon1999} for scalar Jacobi matrices with $A_n \equiv 1$. It was shown in \cite[Theorem 1.1]{LastSimon1999} that the set of $\la \in \RR$ where \eqref{eq:37b} is fulfilled is a~minimal support of the a.c. part of the spectral measure of $J$. Similarly, a variant of it has been studied for bounded block Jacobi matrices in \cite{Oliveira2022} with a similar conclusion. Thus this condition seems to be the right extension to possibly unbounded Jacobi matrices. However, we would like to stress that in the unbounded case the set where \eqref{eq:37b} is satisfied might no longer be a support of the a.c. part of the spectral measure of $J$ even for $d=1$, see Example~\ref{ex:3}.

If the stronger variant of \eqref{eq:37b}, namely,
\begin{equation} \label{eq:37a}
    \limsup_{n \to +\infty} 
    \frac{1}{\rho_n} 
    \sum_{k=1}^{n} \norm{R_k(\la)}^2 < +\infty, \quad \la \in G,
\end{equation}
holds, then we say that $J$ satisfies \emph{Generalized Behncke--Stolz} (GBS in short) condition on $G$. GBS condition has been introduced in \cite[Remark 2.3]{Janas1999} to study spectral properties of scalar unbounded Jacobi matrices. This condition turned out to be very convenient in the study of various classes of unbounded Jacobi matrices, see e.g. \cite{JanasNaboko2001, Janas1999, Moszynski2003, JanasNabokoStolz2004, JanasMoszynski2002}.

Let us start with an observation that the Carleman condition is necessary for GLS to hold. 

\begin{proposition} \label{prop:20}
Suppose that a Jacobi matrix $J$ satisfies \eqref{eq:37b} for some non-empty $G \subset \CC$. Then $J$ is s.a. iff the Carleman condition is satisfied. 
\end{proposition}
\begin{proof}
Let $z \in G$. Then by \eqref{eq:84}
\[
    \begin{pmatrix}
        0 & 0 \\
        Q_n(z) & 0
    \end{pmatrix}
    =
    \begin{pmatrix}
        0 & 0 \\
        0 & \Id
    \end{pmatrix}
    R_n(z)
    \begin{pmatrix}
        \Id & 0 \\
        0 & 0
    \end{pmatrix} 
    \quad \text{and} \quad
    \begin{pmatrix}
        0 & 0 \\
        0 & P_n(z)
    \end{pmatrix}
    =
    \begin{pmatrix}
        0 & 0 \\
        0 & \Id
    \end{pmatrix}
    R_n(z)
    \begin{pmatrix}
        0 & 0 \\
        0 & \Id
    \end{pmatrix}.
\]
Thus,
\[
    \norm{Q_n(z)}^2 + \norm{P_n(z)}^2 
    \leq 
    2 \norm{R_n(z)}^2
\]
and by \eqref{eq:37b}
\[
    \liminf_{n \to \infty}
    \frac{1}{\rho_n} \sum_{k=1}^n 
    \big( \norm{Q_n(z)}^2 + \norm{P_n(z)}^2 \big) < \infty.
\]
Now, if the Carleman condition's is not satisfied, then $\lim_{n \to \infty} \rho_n$ is non-zero and finite, and consequently,
\[
    \liminf_{n \to \infty}
    \sum_{k=1}^n 
    \big( \norm{Q_n(z)}^2 + \norm{P_n(z)}^2 \big) < \infty,
\]
which by \cite[Theorem 1.3]{Dyukarev2020} implies that $J$ is not s.a.
\end{proof}

Below, in Theorem~\ref{thm:5}, we will show that (uniform) GLS condition implies (uniform) $\barrierTR$-nonsubordinacy on $G$ for the barrier \eqref{eq:120}.
\begin{proposition} \label{prop:14}
For any $\la \in \RR$ and any $n \in \NN$ we have
\[
    \frac{\norm{R(\la)}_{[1,n]}^2 }{\lnorm{R(\la)}_{[1,n]}^2} \leq 
    \Bigg( \frac{1}{\rho_n} \sum_{k=1}^n \norm{R_k (\la)}^2 \Bigg)^2.
\]
\end{proposition}
\begin{proof}
From the formula joining $R_k^{-1}(\la)$ and $R_k(\la)$  (see,  e.g., \cite[formula (4.34)]{Weyl}) we get
\[
    \norm{R_k^{-1}(\la)} \leq 
    \norm{A_{k-1}} \norm{R_k(\la)}.
\]
Thus (e.g., using   \cite[formula (2.2)]{Weyl})
\[
    \lnorm{R_k(\la)}^2 = 
    \frac{1}{\norm{R_k^{-1}(\la)}^2} \geq 
    \frac{1}{\norm{A_{k-1}}^2 \norm{R_k(\la)}^2}.
\]
Therefore,
\begin{equation} \label{eq:39}
    \bigg( 
    \frac{1}{\rho_n}
    \sum_{k=1}^n \lnorm{R_k(x)}^2 \bigg)^{-1}
    \leq 
    \bigg( 
    \frac{1}{\rho_n}
    \sum_{k=1}^n \frac{1}{\norm{A_{k-1}}} \frac{1}{\norm{A_{k-1}} \norm{R_k(\la)}^2}
    \bigg)^{-1}.
\end{equation}
Since the mapping $(0, +\infty) \ni x \mapsto x^{-1}$ is convex, Jensen's inequality implies
\[
    \bigg( 
    \frac{1}{\rho_n}
    \sum_{k=1}^n \frac{1}{\norm{A_{k-1}}} \frac{1}{\norm{A_{k-1}} \norm{R_k(\la)}^2}
    \bigg)^{-1}
    \leq
    \frac{1}{\rho_n} \sum_{k=1}^n \frac{1}{\norm{A_{k-1}}} \norm{A_{k-1}} \norm{R_k(\la)}^2
    =
    \frac{1}{\rho_n} \sum_{k=1}^n \| R_k(\la) \|^2,
\]
which combined with \eqref{eq:39} ends the proof.
\end{proof}

\begin{theorem} \label{thm:5}
Let $G \in \Bor(\RR)$ be non-empty. Suppose that $J$ satisfies (uniform) GLS condition on $G$. Then $J$ satisfies (uniform) $\barrierTR$-nonsubordinacy condition on $G$. Consequently, $J$ is absolutely continuous in $G$ and $\clleb{G} \subset \sigmaAC(J)$.
\end{theorem}
\begin{proof}
By Corollary~\ref{cor:4} the function~\eqref{eq:120} is a barrier. Since by Proposition~\ref{prop:14} we have
\[
    \liminf_{t \to +\infty} \barrierTR(\la, t) < +\infty
\]
the operator $J$ satisfies (uniform) $\barrierTR$-nonsubordinacy condition on $G$. Thus, the result follows from Theorem~\ref{thm:10}.
\end{proof}

In the following example we show that even for $d=1$ the uniform barrier nonsubordinacy condition is substantially weaker than GLS.

\begin{example} \label{ex:3}
Let $d=1$. Consider numbers satisfying $\kappa \in (1, \tfrac{3}{2}], f>-1, g>-1$ and $\kappa + 2 g - 2 f < 0$.
Consider Jacobi parameters
\[
    A_n = (n+1)^\kappa \Big( 1 + \frac{f}{n+1} \Big), \quad
    B_n = 2 (n+1)^\kappa \Big( 1 + \frac{g}{n+1} \Big).
\]
Then the Carleman condition for $J$ is {\bf not} satisfied but by \cite[Example 1.3]{jordan2} $J$ is s.a., so $J$ does not satisfy GLS for any non-empty $G \subset \RR$. We shall prove that we can apply Theorem~\ref{thm:10-intr} for any compact interval $K \subset \RR$ with non-empty interior. Indeed, by \cite[Example 1.3]{jordan2} (see also \cite[Theorem B]{jordan2}) there are constants $c_1>0, c_2>0$ such that for any $\lambda \in K$, any $U \in \MGEV{\lambda}$, satisfying $\norm{U_{-1}}^2 + \norm{U_0}^2 = 1$, and any $n \geq 1$
\[
    c_1 \leq \frac{1}{\tilde{\rho}_n} \norm{U}^2_{[0,n]} \leq c_2,
\]
where
\[
    \tilde{\rho}_n = \sum_{k=0}^n \frac{\sqrt{k+1}}{a_k}, \quad n \geq 0.
\]
Thus, for any $U,V \in \MGEV{\lambda}$ satisfying \eqref{eq:144}
\begin{equation} \label{eq:146'}
    \frac{\norm{U}^2_{[0,n]}}{\norm{V}^2_{[0,n]}} 
    \leq 
    \frac{c_2}{c_1}.
\end{equation}
Let us define $\barrier : K \times [1, \infty)$ by $\barrier(\lambda, t) = \frac{c_2}{c_1}$. Then by \eqref{eq:146'} and \cite[Corollary 2.5]{Weyl} the function $\barrier$ is a barrier for $J$ on $K$ and $J$ is uniformly $\barrier$-nonsubordinate on $K$. Therefore, by Theorem~\ref{thm:10-intr} \  $J$ is absolutely continuous on $K$ and $K \subset \sigmaAC(J)$. Since $K$ was arbitrary, we get that $J$ is absolutely continuous on $\RR$ and $\sigmaAC(J) = \RR$.
\end{example}

\subsection{Homogenous class ($H$ class) condition} \label{sec:6.3}
For the arbitrary size $m\times m$ of matrices this class is given as follows (see, e.g.,  
\cite[Definition 2.3]{Moszynski2019}).
\begin{definition}
Let $(C_n)_{n \geq n_0}$ be a sequence of complex $m \times m$ matrices. 
We say that $(C_n)_{n \geq n_0} \in H$ if there is a constant $c>0$ such that for any $n \geq n_0$
\begin{equation} \label{eq:14}
    \| C_n \cdot C_{n-1} \cdot \ldots \cdot C_{n_0} \| 
    \leq 
    c \prod_{k=n_0}^n |\det C_k|^{1/m}.
\end{equation}
\end{definition}

The concept of $H$ class was introduced in \cite{Moszynski2003} (see also \cite[Lemma 1.8]{JanasMoszynski2002}) as a sufficient condition for GBS for scalar Jacobi matrices. It turned out that sometimes this stronger condition is somewhat easier to show than to show ``ditectly'' GBS. Later, in \cite{Moszynski2019} this notion has been extended to block Jacobi matrices. An important consequence of Theorems 1.7 and 2.7 from \cite{Moszynski2019} is that the sequence of transfer matrices $\big( T_n(\la) \big)_{n \in \NN_0}$ belongs to $H$ iff there exists a constant $c > 0$ such that
\begin{equation} 
    \label{eq:15b}
    c \norm{R_n(\lambda)} \leq \lnorm{R_n(\lambda)}, \quad n \geq 1.
\end{equation}

In the following we show that in general $H$ class implies $\barrierTR$-nonsubordinacy condition, i.e., with the transfer matrix barrier \eqref{eq:120}, which solves affirmatively the open problem \cite[Remark 2.8.4]{Moszynski2019}.
 \begin{theorem} \label{thm:8}
 Suppose that $J$ is s.a. 
 Let $G \in \Bor(\RR)$ be non-empty. Suppose that $\big( T_n(\la) \big)_{n \in \NN_0} \in H$ for any $\la \in G$. Then $J$ satisfies $\barrierTR$-nonsubordinacy condition on $G$. Consequently, $J$ is absolutely continuous in $G$ and $\clleb{G} \subset \sigmaAC(J)$.

 \end{theorem}
 \begin{proof}
 Observe that by \eqref{eq:15b} there is $c(\la) > 0$ such that for any $k \geq 1$ we have $c(\la) \norm{R_k(\la)}^2 \leq \lnorm{R_k(\la)}^2$. Thus, for any $t \geq 1$ we get
 \[
     c(\la) \norm{R(\la)}_{[1,t]}^2 
     \leq
     \lnorm{R(\la)}_{[1,t]}^2.
 \]
 In other words,
 \[
     \frac{\norm{R(\la)}_{[1,t]}^2 }{\lnorm{R(\la)}_{[1,t]}^2} \leq \frac{1}{c(\la)}
 \]
 Thus
 \[
     \limsup_{t \to +\infty} \barrierTR(\la, t) \leq 
     \frac{8}{c(\la)} < +\infty.
 \]
 Consequently, the operator $J$ satisfies $\barrierTR$-nonsubordinacy condition on $G$. Thus, the result follows from Theorem~\ref{thm:10}.
\end{proof}

\section{Examples and applications} \label{sec:7}
In this section we show some examples and counterexamples illustrating the applicability of our results.
 
\subsection{''The invertibility of the density of $M$'' does not guarantee the nonsubordinacy}
\label{susec:7-1}

Let us observe that direct sums of purely absolutely continuous scalar Jacobi matrices with different asymptotic behaviour of their generalized eigenvectors does not satisfy our nonsubordinacy condition. Below, we shall present a simple example of this kind.
\begin{example} \label{ex:2}
Consider diagonal $2\times 2$ blocks defining  BJO $J$ (the maximal one):
\[
    A_n =
    \begin{pmatrix}
        a_n^{(1)} & 0 \\
        0 & a_n^{(2)}
    \end{pmatrix},
    \quad
    B_n =
    \begin{pmatrix}
        b_n^{(1)} & 0 \\
        0 & b_n^{(2)}
    \end{pmatrix},
\]
where $a^{(i)} = \big( a_n^{(i)} \big)_{n \in \NN_0}, b^{(i)} = \big( a_n^{(i)} \big)_{n \in \NN_0}$ are Jacobi parameters of  two scalar Jacobi matrices $\calJ^{(i)}$ for $i=1,2$, and  let $J^{(i)}$ be the appropriate  (also maximal) JO. 
Let us choose
\[
    a_n^{(i)} = (n+1)^{\alpha_i}, \quad 
    b_n^{(i)} = 0, \quad n \geq 0, \ i=1,2,
\]
where $\alpha_1,\alpha_2 \in (0,1]$ and $\alpha_1 > \alpha_2$.
Then both $J^{(i)}$ are s.a. Moreover,  by \cite[Theorem 1]{PeriodicII},  they are absolutely continuous,  $\sigmaAC(J^{(i)})=\RR$, $\sigmaS(J^{(i)})=\emptyset$ for  $i=1,2$,  the corresponding spectral measures $\mu^{(i)}$ of $J^{(i)}$ are absolutely continuous   and  continuous positive densities w.r.t.  the Lebesgue measure can be chosen for them. Thus, $J$ is s.a., $J$ is absolutely continuous, $\sigmaAC(J)=\RR$, $\sigmaS(J)=\emptyset$,
and, it can be shown that the  matrix measure $M$ is a diagonal matrix with the measures $\mu^{(1)}$ and $\mu^{(2)}$ on the diagonal. In particular, there exists a density of $M$ which is strictly positive on $\RR$.

 Fix $\la \in \RR$. Let us  show that $J$ does not satisfy vector nonsubordinacy condition for  $\la$. Let us choose a non-zero $u^{(i)} \in \GEV{J^{(i)}, \la}$ for $i=1,2$. Set $u := \big( u^{(1)}_n e_1 \big)_{n \in \NN_{-1}}$ and $v := \big( u^{(2)}_n e_2 \big)_{n \in \NN_{-1}}$. Then $u,v \in \GEV{J, \la}$. According to \cite[Theorem A]{PeriodicI} 
\begin{equation} 
    \label{eq:151}
    \frac{\norm{u}_{[0,n]}^2}{\norm{v}_{[0,n]}^2} \asymp_n
    \frac{\sum_{k=0}^n \frac{1}{a_n^{(2)}}}{\sum_{k=0}^n \frac{1}{a_n^{(1)}}},
\end{equation}
where the symbol $\asymp_n$ above means that the ratio of both sides can be bounded from above and from below by some real positive constants  for all $n$s. 
By Stolz--Ces\`{a}ro lemma
\begin{equation}
    \label{eq:152}
    \lim_{n \to +\infty}
    \frac{\sum_{k=0}^n \frac{1}{a_n^{(2)}}}{\sum_{k=0}^n \frac{1}{a_n^{(1)}}} =
    \lim_{n \to +\infty} \frac{a_n^{(1)}}{a_n^{(2)}} = +\infty.
\end{equation}
Therefore, by \eqref{eq:151} and \eqref{eq:152} and \cite[Corollary 2.5]{Weyl} it easily follows that $J$ does not satisfy  nonsubordinacy condition for~$\la$ (see \cite[subsection  6.1]{Weyl} for some  more detailed discussion about the equivalence of the nonsubordinacy and the vector nonsubordinacy).
\end{example}

\subsection{Application to some classes of block Jacobi operators} \label{sec:appl-BJM}
In this section we would like to illustrate our results by showing that Jacobi operators considered in \cite[Theorem 2]{block2018} are in fact absolutely continuous in some explicit region of the real line.

Let us introduce a necessary notion. Let $X = (X_n)_{n \in \NN}$ be a sequence from a normed space $V$. Let $N \geq 1$ be a positive integer. We say that $X \in \calD_1^N$ if
\[
    \sum_{n=1}^{+\infty} \norm{X_{n+N} - X_n}_V < +\infty.
\]
Observe that if $X \in \calD_1^N$, then for any $i \in \{0, 1, \ldots, N-1 \}$ the sequence $(X_{kN+i})_{k \in \NN}$ is a Cauchy sequence.

\begin{theorem} \label{thm:11}
Let $J$ be a block Jacobi operator. Suppose that Jacobi parameters of $J$ satisfy for some integer $N \geq 1$
\begin{equation} \label{eq:100}
    \big( A_n^{-1} \big)_{n \in \NN},
    \big( A_n^{-1} B_n \big)_{n \in \NN},
    \big( A_n^{-1} A_{n-1}^* \big)_{n \in \NN} \in \calD_1^N.
\end{equation}
Then for any $i \in \{0, 1, \ldots, N-1 \}$ and any $z \in \CC$ the limit
\begin{equation} \label{eq:99}
    \calX_i(z) :=
    \lim_{\substack{n \to +\infty\\n \equiv i \bmod{N}}} T_{n+N-1}(z) \ldots T_{n+1}(z) T_n(z)
\end{equation}
exists, where $T_n$ is defined in \eqref{eq:21}. Suppose further that
for some $N$-periodic sequence of invertible matrices $(\calC_n)_{n \in \NN_0}$
\[
    \lim_{n \to +\infty} \norm{\frac{a_n}{\norm{a_n}} - \calC_n} = 0.
\]
Define a matrix
\[
    Q(z) = 
    \Re 
        \bigg[
        \begin{pmatrix}
            0 & -\calC_{N-1} \\
            \calC_{N-1}^* & 0
        \end{pmatrix}
        \calX_0(z) 
    \bigg], \quad z \in \CC
\]
and let
\[
    \Lambda = 
    \bigg\{ 
        \la \in \RR : Q(\lambda)
         \text{ is strictly positive or strictly negative on } \CC^{2d}
    \bigg\}.
\]
Then $\Lambda \subset \RR$ is open and $\big( T_n(\la) \big)_{n \in \NN} \in H$ 
for any $\la \in \Lambda$. Consequently, if Carleman condition is satisfied, then $J$ is absolutely continuous in $\Lambda$ and $\cl{\Lambda} \subset \sigmaAC(J)$.
\end{theorem}
\begin{proof}
First of all, let us observe that \eqref{eq:100} implies that for any $j \in \{0, 1, \ldots, N-1\}$ the sequence $(T_{kN+j})_{k \in \NN}$ is convergent. Thus, the limit \eqref{eq:99} exists. Since every entry of the product $T_{n+N-1}(z) \ldots T_{n}(z)$ is a polynomial of degree at most $N$, the limit is even locally uniform on $\CC$ and the function $\calX_i$ is continuous. So, the entries of the Hermitian matrix $Q$ are continuous as well. Let us recall that according to Sylvester's criterion the positivity/negativity of $Q(z)$ is determined by the appropriate signs of the principal minors of $Q(z)$, which is an open condition. Therefore, $\Lambda$ is an open subset of $\RR$. 

We are going to show that the condition~\eqref{eq:15b} is satisfied. Let us recall that according to \cite[Theorem 2]{block2018} for any compact interval $K \subset \Lambda$ there are constants $c_1 > 0, c_2 > 0$ such that for any normalized $u \in \GEV{\la}$, where $\la \in K$, and any $n \geq 1$ we have
\begin{equation} \label{eq:101}
    \frac{c_1}{\norm{A_n}} \leq 
    \norm{
    \begin{pmatrix}
        u_{n-1} \\
        u_n
    \end{pmatrix}
    }_{\CC^{2d}}^2
    \leq 
    \frac{c_2}{\norm{A_n}}.
\end{equation}
Recall that for any $u \in \GEV{\la}$
\[
    \begin{pmatrix}
        u_{n-1} \\
        u_n
    \end{pmatrix}
    = 
    R_n(\la)
    \begin{pmatrix}
        u_{-1} \\
        u_0
    \end{pmatrix}.
\]
Hence, having the notation above in mind we have
\[
    \lnorm{R_n(\la)}^2
    =
    \inf_{\substack{\norm{u_{-1}}_{\CC^d}^2 + \norm{u_0}_{\CC^d}^2 = 1\\u \in \GEV{\la}}} 
    \norm{
    \begin{pmatrix}
        u_{n-1} \\
        u_n
    \end{pmatrix}
    }_{\CC^{2d}}^2, \quad
    \norm{R_n(\la)}^2
    =
    \sup_{\substack{\norm{u_{-1}}_{\CC^d}^2 + \norm{u_0}_{\CC^d}^2 = 1\\u \in \GEV{\la}}} 
    \norm{
    \begin{pmatrix}
        u_{n-1} \\
        u_n
    \end{pmatrix}
    }_{\CC^{2d}}^2.
\]
Thus, \eqref{eq:101} implies
\begin{equation} \label{eq:143}
    \frac{c_1}{\norm{A_n}} \leq 
    \lnorm{R_n(\la)}^2 \leq 
    \norm{R_n(\la)}^2 \leq
    \frac{c_2}{\norm{A_n}},
\end{equation}
which shows that \eqref{eq:15b} is satisfied. Therefore, $\big( T_n(\la) \big)_{n \in \NN} \in H$ for any $\la \in \Lambda$. Finally, if the Carleman condition is satisfied, then by Theorem~\ref{thm:8}, $J$ is absolutely continuous on $\Lambda$ and $\clleb{\Lambda} \subset \sigmaAC(J)$. Since $\Lambda$ is open, by Proposition~\ref{cor:1}\ref{cor:1:2}, we get that $\clleb{\Lambda} = \cl{\Lambda}$. The proof is complete. 
\end{proof}

\begin{bibliography}{block,jacobi}
	\bibliographystyle{amsplain}
\end{bibliography}
\end{document}